\documentclass{amsart}
\usepackage[margin=1in]{geometry}
\usepackage{amsmath, amsfonts, amssymb, amsthm}
\usepackage{graphicx}
\usepackage[usenames,dvipsnames,svgnames,table]{xcolor}
\usepackage[colorlinks=true, pdfstartview=FitV,linkcolor=ForestGreen,citecolor=ForestGreen, urlcolor=blue]{hyperref}
\usepackage{url}
\usepackage{xcolor}
\usepackage{dsfont}
\usepackage{esint}

\newtheorem{thm}{Theorem}[section]
\newtheorem{prop}[thm]{Proposition}
\newtheorem{cor}[thm]{Corollary}

\newtheorem{lemma}[thm]{Lemma}
\newtheorem{defn}[thm]{Definition}

\newtheorem{preremark}[thm]{Remark}
\newenvironment{remark}{\begin{preremark}\rm}{\medskip \end{preremark}}
\numberwithin{equation}{section}

\newcommand{\R}{\mathbb R}
\newcommand{\eps}{\varepsilon}

\newcommand{\dd} {\; \mathrm{d}}

\newcommand{\cOne}{c_K}
\newcommand{\cP}{C_P}

\newcommand{\IB}{\mathcal{I}_\B}
\newcommand{\hs}{\mathcal{S}}
\newcommand{\is}{\mathcal{I}}
\newcommand{\js}{\mathcal{J}}

\newcommand{\bb}{\mathfrak{b}}

\newcommand{\db}{b'}
\newcommand{\dalpha}{\alpha'}

\newcommand{\B}{\mathcal B}

\title{On the monotonicity of the Fisher information for the Boltzmann equation}
\date{\today}
\author{Cyril Imbert}
\address{Département de mathématiques et applications, École normale supérieure, Université PSL, CNRS, 75005 Paris, France}
\email{cyril.imbert@ens.fr}

\author{Luis Silvestre}
\address{Mathematics Department, University of Chicago, Chicago, IL 60637, USA}
\email{luis@math.uchicago.edu}

\author{C\'edric Villani}
\address{Université Claude Bernard Lyon 1, Institut Camille Jordan, 69622 Villeurbanne Cedex, and Institut des Hautes Études Scientifiques (IHES), Bures sur Yvette. France}
\email{cv@cedricvillani.org}

\thanks{The authors warmly thank Cl\'ement Mouhot for fruitful discussions during the preparation of this paper. The authors also thank the Conference and Summer school Mathemata / Festum Pi held in Chania, Crete, July 2024, where part of this research was performed. Luis Silvestre is supported by NSF grants DMS-2054888 and DMS-2350263}

\begin{document}

\begin{abstract}
We prove that the Fisher information is monotone decreasing in time along solutions of the space-homogeneous Boltzmann equation for a large class of collision kernels covering all classical interactions derived from systems of particles. For general collision kernels, a sufficient condition for the monotonicity of the Fisher information along the flow is related to the best constant for an integro-differential inequality for functions on the sphere, which belongs in the family of the Log-Sobolev inequalities. As a consequence, we establish the existence of global smooth solutions to the space-homogeneous Boltzmann equation in the main situation of interest where this was not known, namely the regime of very soft potentials. This is opening the path to the completion of both the classical program of qualitative study of space-homogeneous Boltzmann equation, initiated by Carleman, and the program of using the Fisher information in the study of the Boltzmann equation, initiated by McKean. From the proofs and discussion emerges a strengthened picture of the links between kinetic theory, information theory and log-Sobolev inequalities.
\end{abstract}

\maketitle

\setcounter{tocdepth}{1}
\tableofcontents

\section{Introduction}

\subsection{The space-homogeneous Boltzmann equation}

We consider the space-homogeneous Boltzmann equation in  $\R^d$,
\begin{equation} \label{e:boltzmann}
\partial_t f = q(f),
\end{equation}
where the collision operator is defined by
\[ q(f) = \int_{w \in \R^d} \int_{\sigma \in \R^{d-1}} (f(w')f(v') - f(w) f(v)) B(|v-w|,\cos\theta) \dd \sigma \dd w.\]
We recall that the function  $B \colon [0,+\infty) \times [-1,1] \to (0,+\infty)$ is known as the \emph{collision kernel} of the operator and velocities $v'$ and $w'$ are given by the formulas
\[ v' = \frac{v+w}2 + \frac{|v-w|}2 \sigma \quad \text{ and } \quad w' = \frac{v+w}2 - \frac{|v-w|}2 \sigma
  \quad \text{ and } \quad \cos \theta = \frac{v-w}{|v-w|} \cdot \frac{v'-w'}{|v'-w'|}.\]

A common class of collision kernels $B$ considered in the literature, ever since Maxwell pioneered the field, is written as a product of two nonnegative functions $B(|v-w|,\cos\theta) = \alpha(|v-w|/2) \, b(\cos\theta)$. This class includes hard spheres and inverse power-law potentials (force proportional to the power $q$ of the distance between molecules, $q\geq 2$):
\begin{itemize}
\item For hard-spheres: $\alpha(r) = r$ and $b(\cos \theta) = (\sin(\theta/2))^{-(d-3)}$; in particular, $b \equiv 1$ in three dimensions.
\item For inverse power-law potentials: $\alpha(r) = r^\gamma$ and $b(c) \approx (1-c^2)^{-(d-1+2s)/2}$, for $\gamma = (q-2d+1)/(q-1)$ and $2s = (d-1)/(q-1)$.
\end{itemize}

The case of Coulomb interaction does not, strictly speaking, fall in the realm of the Boltzmann equation, because the angular singularity is too strong. The obstruction here is physical as well as mathematical: In fact, the most generic condition for the operator $q(f)$ to make sense is
\begin{equation} \label{e:generic-b}
\forall e \in S^{d-1}, \qquad \int_{S^{d-1}} (1-(e\cdot \sigma)^2) b(e\cdot \sigma) \dd \sigma < +\infty.
\end{equation}
(Note that the left-hand side of \eqref{e:generic-b} is independent of the choice of $e \in S^{d-1}$.) In the case of Coulomb interaction in dimension 3, the integral above is borderline divergent. This was resolved by Landau through a diffusive approximation of the Boltzmann equation, which is important both for theoretical and practical applications in plasma physics and more generally kinetic theory \cite{villani-crete}.

It was in the 1930's that Carleman initiated the mathematical study of the space-homogeneous Boltzmann equation, asking for the basic properties of localization (decay at infinity), integrability, regularity, and positivity (lower bound). While the assumption of space-homogeneity is of course a huge simplification, it is a legitimate, and even mandatory, step to understand the properties of the Boltzmann equation and in particular the complexity of the Boltzmann operator itself. Carleman \cite{carleman} proved the existence of classical solutions for hard spheres in dimension 3, under some symmetry assumptions, and established their convergence to Maxwellian (Gaussian) equilibrium in large time. 

This program was then carried and extended by many researchers, covering more and more singular situations. Arkeryd \cite{zbMATH03388763} pioneered the systematic study of hard potentials with cutoff, and later \cite{ark:infini:81} hard or moderately soft potentials with cutoff. Bobylev \cite{bob:theory:88} tackled Maxwellian molecules with considerable depth. Goudon \cite{goud:graz:97} and the third author \cite{zbMATH01221506} considered more general soft and very soft potentials, respectively. Over the years, a number of regularity techniques have been introduced and refined to tackle these situations. These include moment inequalities, integrability and regularity of Boltzmann's gain operator, Fourier analysis of the angular singularity, integral regularity estimates, iterative schemes in Moser style, maximum principles of various kinds, Harnack inequalities for jump processes. Far from being exhaustive, among contributors we may mention Arkeryd, Desvillettes, Gustafsson, Lions, Mischler, Mouhot, Wennberg and the three authors themselves.
A review of these methods before 2002 can be found in \cite{villani-survey}; further references will be given along the way. 

In spite of this, regularity in the most singular range, that is, very soft potentials, has long remained elusive, forcing authors to focus on ad hoc notions of time-integrated weak solutions involving entropy production \cite{ADVW,carlen2009,golse2023partialregularitytimespacehomogeneous,zbMATH01221506}.
It was known that $L^p$ integrability for $p$ large enough would yield regularity, but the integrability estimate was nowhere to be found with known tools, leaving experts to wonder about the possibility of regularity blow up.

We are, in this paper, precisely closing this gap by establishing a new a priori estimate in $L^3 (\R^3)$, or more generally $(L^{d/(d-2)}(\R^d))$ ($L^p$ for any $p$ if $d=2$), for a wide class of collision kernels, factorized or not, actually including all classical inverse power-law potentials: hard, soft, very soft, cutoff or noncutoff, and hard spheres, in any dimension. This will be obtained by showing a result of mathematical and physical interest in itself: For a wide class of collision kernels, the usual Fisher information
\[ i(f) := \int_{\R^d} \frac{|\nabla f|^2} f \dd v\]
is monotone decreasing in time along solutions of the space-homogeneous Boltzmann equation \eqref{e:boltzmann}.

Not only does this result allow us to envision the completion of Carleman's program on the qualitative analysis of the space-homogeneous Boltzmann equation, but it is also, in some sense, the completion of a program initiated by McKean \cite{mckean1966} on Kac's one-dimensional caricature of the Boltzmann equation. Before saying more on this, let us describe our results more precisely. This will lead us to a detour into the theory of entropic inequalities for diffusion processes.

\subsection{Monotonicity of the Fisher information}

Let us start with a simple informal statement about interaction potentials proportional to the inverse power $q$ of the intermolecular distance, with $q \geq d-1$ (that is, decaying at least as fast as  $d$-dimensional Coulomb) and $q> (d+1)/2$ (to ensure the convergence of the cross-section for momentum transfer, \eqref{e:generic-b}).

\begin{thm}[The Fisher information decreases along the Boltzmann flow] \label{t:main-simpler}
Let $B$ be the collision kernel for hard-spheres, or for any of the physically relevant inverse power-law potentials, then the Fisher information of a solution $f$ to \eqref{e:boltzmann} is nonincreasing in time.
\end{thm}

Our main result applies to a much wider class of kernels, factorized or not, under some condition. In order to explain it, we first consider the factorized case in which $B$ is the product of $\alpha(r/2)$ with an angular cross-section, $b(\cos\theta)$. Keeping only the angular variation of the kernel, define the spherical linear Boltzmann operator $\B$, acting on functions $f \colon S^{d-1} \to \R$, by
\begin{equation} \label{e:spherical-operator}
\B f(\sigma) := \int_{S^{d-1}} \left( f(\sigma') - f(\sigma) \right) b(\sigma' \cdot \sigma) \dd \sigma'.
\end{equation}
This integro-differential operator defines a diffusion semigroup, and there is a corresponding \emph{carré du champ} operator in the terminology of Bakry-\'Emery \cite{zbMATH03894218}. Let $\Gamma^2_{\B,\Delta}$ be the iterated carré du champ operator obtained by differentiation first along the flow of the usual heat equation and then along the flow of $\B$ (see Section \ref{s:carré-du-champ} for a precise definition and practical formulas, Lemmas~\ref{l:Gamma2-with-T} and \ref{l:Gamma2-with-vector-fields}). Let $\Lambda_b \geq 0$ be the largest constant such that the following functional inequality holds: For every function $f \colon S^{d-1} \to (0,\infty)$ such that $f(\sigma) = f(-\sigma)$, 
\begin{equation} \label{e:spherical-inequality}
\begin{aligned} \int_{S^{d-1}} & \Gamma^2_{\B,\Delta} (\log f,\log f) f \dd \sigma \geq \Lambda_b \iint_{S^{d-1} \times S^{d-1}} \frac{(f(\sigma') - f(\sigma))^2} {f(\sigma')+f(\sigma)} b(\sigma' \cdot \sigma) \dd \sigma ' \dd \sigma.
 \end{aligned}
\end{equation}
Obviously the value of $\Lambda_b$ is a property of $b$; we shall provide recipes to evaluate $\Lambda_b$, and examples in which it is positive and other ones for which it is zero. But before we move on to this, let us see how the constant $\Lambda_b$ allows us to formulate our monotonicity theorem.

\begin{thm}[The Fisher information decreases along the Boltzmann flow] \label{t:main}
Assume that $\alpha$ and $b$ are such that, for all  $r>0$,
\[ \frac{r |\dalpha (r)|}{ \alpha(r)} \leq 2\sqrt{\Lambda_b},\]
then the Fisher information of any solution $f$ to \eqref{e:boltzmann} with collision kernel $B = \alpha(r/2) b(\cos\theta)$ is nonincreasing in time.
\end{thm}

\begin{remark}
Note that in the common case that $\alpha(r) = r^\gamma$, the left hand side of the assumption of Theorem~\ref{t:main} is $r |\dalpha (r)| / \alpha(r) = |\gamma|$. Therefore, the assumption of Theorem~\ref{t:main} takes the form of a restriction on the range of $\gamma \in [-2\sqrt{\Lambda_b},2\sqrt{\Lambda_b}]$.
\end{remark}

\begin{remark}
In the special case $d=2$, the inequality \eqref{e:spherical-inequality} simplifies and can be stated in an elementary way: It says that for every function $f \colon \R \to (0,\infty)$ that is $\pi$-periodic, we have
\begin{equation} \label{e:inequality2D}
\begin{aligned} \iint_{[0,\pi] \times [0,\pi/2]} &\bigl |\partial_\theta \log f(\theta+h) - \partial_\theta \log f(\theta)\bigr|^2 f(\theta) b(h) \dd \theta \dd h \\ \qquad &\geq \Lambda_b \iint_{[0,\pi] \times [0,\pi/2]} \frac{(f(\theta+h) - f(\theta))^2} {f(\theta+h)+f(\theta)} b(h) \dd \theta \dd h.
 \end{aligned}
\end{equation}
\end{remark}

To state our more general results, it will suffice to apply the previous criterion to the angular cross-section $b(\cos \theta) = B(r,\cos \theta)$, for each fixed value of $r > 0$: Explicitly,

\begin{thm}[The Fisher information decreases along the Boltzmann flow, more general version] \label{t:main:gen}
Assume that $B=B(r,\cos\theta)$ satisfies, for all $r>0$ and $\theta\in [0,\pi]$,
\[ \frac{r |\partial_r B(r,\cos \theta)|}{B(r,\cos \theta)} \leq 2 \sqrt{\Lambda_{B(r,\cdot)}}.\]
Then the Fisher information of any solution $f$ to \eqref{e:boltzmann} with collision kernel $B$ is nonincreasing in time.
\end{thm}

The proof of Theorem \ref{t:main:gen} is the same as that of Theorem \ref{t:main}, through very simple changes and a bit of cumbersome notation, so we shall focus on just Theorem \ref{t:main}. The proof adapts the main ideas of \cite{landaufisher2023} from the purely diffusive setting of the Landau equation with Coulomb interaction, to the integro-differential setting of the Boltzmann equation. In doing so, we face a major new challenge, namely to explicitly estimate $\Lambda_b$ from below for concrete kernels $b$. We will explain how this problem is related to integro-differential versions of the log-Sobolev inequality on the sphere, or more properly on its quotient by $\{\pm {\rm Id}\}$, the real projective space -- an area of research that is connected to a large body of literature but which in itself had remained relatively unexplored. We will also give two complementary recipes to address this estimate in concrete settings. Along the way, dimension and curvature effects will appear. Here is a main example:

\begin{itemize}
\item In dimension $d=2$ there are singular kernels such that $\Lambda_b=0$;

\item In dimensions 3 and higher, for any kernel $b$, singular or not, it holds $\Lambda_b>0$.
\end{itemize}

These results are proved  in Section \ref{s:counterexample} and Proposition \ref{p:spherical-inequality-from-curvature}, respectively. Estimating $\Lambda_b$ explicitly and with fair degree of precision will turn out to be the most difficult part of the analysis in this paper. We shall establish in particular that:

\begin{itemize}
	\item When $d \geq 3$, for any kernel satisfying merely \eqref{e:generic-b}, \eqref{e:spherical-inequality} holds with \[ \Lambda_b \geq d-2. \]
	This is proved in Proposition \ref{p:spherical-inequality-from-curvature}.
	\item If $b$ is a constant kernel (and in particular for hard spheres in dimension three), we prove that \eqref{e:spherical-inequality} holds with $\Lambda_b \geq d$.
	This is proved in Proposition~\ref{p:hard-spheres}.
	\item If $b$ is the hard sphere kernel in dimension 2, then $\Lambda_b \geq \sqrt{2}$. This is explained in Subsection~\ref{subsec:comments}.
	\item For any dimension, assume that the kernel $b$ can be written in the form
	\begin{equation} \label{e:intro-subordinate} b(c) = \int_0^\infty h_t(c) \omega(t) \dd t, \end{equation}
	where $h_t$ is the heat kernel on the sphere. In this case we show that \eqref{e:spherical-inequality} holds with $\Lambda_b > d$ (Proposition \ref{p:main-for-subordinate}).
	\item Assume that $b_0$ is a reference kernel for which \eqref{e:spherical-inequality} holds with parameter $\Lambda_0$. Let $b$ be another kernel such that the following inequality holds for all $c \in [-1,1]$,
\begin{equation} \label{e:kernel-comparison}
	c_1 \bigg(b_0 (c) + b_0 (-c) \bigg) \leq b(c) + b(-c)  \leq C_2 \bigg( b_0 (c) + b_0 (-c) \bigg).
	\end{equation}
	Then \eqref{e:spherical-inequality} holds for $b$ with $\Lambda_b \geq c_1 \Lambda_0/C_2.$
\end{itemize}

\subsection{Comments on the class of kernels satisfying our assumptions} \label{subsec:comments}

First about high dimensions. In dimensions 4 and higher, hard spheres, as well as inverse power-law kernels decaying at least as fast as $d$-dimensional Coulomb interaction, are all covered by Proposition \ref{p:spherical-inequality-from-curvature}. Then we are left with dimensions 2 and 3.
\bigskip

Next about low-dimensional hard spheres. For $d=3$ they are covered both by Proposition \ref{p:spherical-inequality-from-curvature} and by Proposition \ref{p:hard-spheres}. For $d=2$ the hard spheres kernel is $b_{hs}(\cos \theta) = \sin(\theta/2)$ and $\alpha_{hs}(r) = r$. Since inequality~\ref{e:spherical-inequality} is only required for even functions, $\Lambda_b$ is unchanged upon replacement of $b$ by its symmetric version $[b(c)+b(-c)]/2$. But then it is readily seen that
\[ 1 \leq b_{hs}(c) + b_{hs}(-c) \leq \sqrt 2.\]
Thus, we can apply \eqref{e:kernel-comparison} to compare the kernels $b_{hs}$ and $b_{constant}$. For constant kernels in 2D, we have $\Lambda_{constant} \geq 2$, and then the inequality \eqref{e:spherical-inequality} holds for $b_{hs}$ with parameter $\Lambda_{b_{hs}} \geq \sqrt 2$, and indeed $2\sqrt{\Lambda_b} >1$.
\bigskip

Now, the class of kernels $b$ that are written as an integral of a nonnegative weight $\omega(t)$ times the heat kernel on the sphere as in \eqref{e:intro-subordinate} is relatively common in probability theory. These are the jump kernels of Lévy processes constructed from subordination of Brownian motion. For example, the kernel corresponding to the fractional Laplacian $(-\Delta)^s$ on the sphere takes the form \eqref{e:intro-subordinate} with the weight $c_s t^{-1-s}$ (for some $c_s>0$ depending on $s$ and dimension).
\bigskip

Another physically relevant class of collision kernels are those that correspond to a repulsive potential with a power law with exponent $q \in (2,\infty)$ in three dimensions. In this case, the kernel $B$ has the form $B(r,\cos \theta) = b(\cos \theta) r^\gamma$. The exponent is $\gamma = (q-5)/(q-1)$. The angular part $b(\cos \theta)$ does not have an explicit formula. It has a nonintegrable singularity as $\theta\to 0$ of the form $b(\cos \theta) \approx \theta^{-2-2s}$ for $s = 1/(q-1)$, similarly to the kernel of $(-\Delta)^s$. Its precise description can be found in \cite{cercignani-1988} or \cite{villani-survey}. 
When $q=2$, a classical computation  due to Rutherford \cite{rutherford1911lxxix} yields, 
 \[ b (\cos \theta) =  \sin (\theta /2)^{-2(d-1)}.\]
 Remarquably, it coincides in dimension $d=2$ with the kernel of the fractional Laplacian $(-\Delta)^{\frac12}$ restricted to $\pi$-periodic functions (see for instance \cite[Proposition~3]{MR2854312}).
 
In general, the exact formula for the angular kernel $b$ is only defined implicitly in terms of certain complicated integrals.

When $s \in [0,3/4]$ (which corresponds to $q \geq 7/3$), we have $\gamma \in [-2,1]$. In this range, Theorem \ref{t:main} applies since $\Lambda_b \geq 1$ for just any kernel $b$ in three dimensions, and $|\gamma| \leq 2$.

It is not absurd to expect the power-law kernels to belong to the class of subordinate kernels, since it is a smooth function away from $\theta=0$ with a similar singularity as the fractional Laplacian. But we have no clue how to attack such a conjecture. So we gave ourselves a more modest task, and still sufficient for applications to the Boltzmann equation, which is to numerically check that these inverse power law kernels remain close enough to kernels corresponding to a subordinate function. It turns out that such is the case: If $c_1$ and $C_2$ are the optimal constants in \eqref{e:kernel-comparison}, we found empirically, for all values of $q$ in the relevant range, a weight function $\omega$ with a ratio $c_1/C_2$ larger than $0.9$. The value of $\Lambda_b$ can be computed for this weight function $\omega$ using the formula in Proposition \ref{p:main-for-subordinate}. Our computations show convincingly that the assumption of Theorem \ref{t:main} is satisfied by a confortable margin for all power-law kernels in the full range of parameters in 3D and 2D. We review our numerical computations in Appendix~\ref{s:numerical}.

\bigskip

Summarizing, while Theorem \ref{t:main} requires a certain assumption on the collision kernel $B$, this assumption is satisfied by all the common scenarios of physical relevance, including the whole range of 3D inverse power law forces and hard spheres, and also the whole classical range in other dimensions.

\subsection{Global smooth solutions of the Boltzmann equation with very soft potentials}

In \cite{landaufisher2023}, the monotonicity of the Fisher information for the Landau equation is used to prove that there exist global-in-time smooth solutions. In a similar way, we establish here the existence of global smooth solutions for the homogeneous non-cutoff Boltzmann equation with very soft potentials. It is a consequence of the monotonicity of the Fisher information proved in our Theorem~\ref{t:main} and Sobolev embedding, together with existing results in the literature. Thus, the solutions of the space-homogeneous non-cutoff Boltzmann equation, under the assumption of Theorem~\ref{t:main}, cannot blow up even in the very soft potential range.

For simplicity we consider a collision kernel of the form $B(r,\cos \theta) = \alpha(r/2) b(\cos \theta)$. The typical non-cutoff assumptions are that
\begin{equation} \label{e:non-cutoff-standard} 
\alpha(r) = r^\gamma \qquad \text{and} \qquad b(\cos \theta) \approx |\sin\theta|^{-d+1-2s}.
\end{equation}

The existence of global smooth solutions for the space-homogeneous Boltzmann equation with $\gamma+2s \geq 0$ is currently well understood (see for example \cite{he2012}). In the case of very soft potentials (i.e. $\gamma+2s<0$), the possibility of finite time blow up has been a well known open problem for some time. Our monotonicity theorem resolves this problem.

\begin{thm}[Global existence for very soft potentials] \label{t:global_existence}
Let $d=3$ and $B$ be a collision kernel of the form \eqref{e:non-cutoff-standard} with $s \in (0,1)$ and $\gamma \in (-d,0]$. Assume that $B$ satisfies the hypothesis of Theorem \ref{t:main}. There exists an exponent $q$ (large enough) so that if $f_0$ is any initial data in $L^\infty_q(\R^3)$ the Boltzmann equation \eqref{e:boltzmann} has a global classical solution in $[0,\infty) \times \R^3$ with initial data $f_0$.
\end{thm}

The restriction to three dimensions, as well as the precise assumption $f_0 \in L^\infty_q(\R^3)$, are the assumptions of the short-time existence results that we reference from the literature. Any advancement in the short-time well posedness theory would immediately be reflected in an enhancement of Theorem \ref{t:global_existence} to a larger class of initial data. While we do not wish to be overoptimistic, everything seems to indicate that present-day tools are enough to work out more general dimensions as well as relaxed conditions, say $f_0 \in L^p_q(\R^d)$ for $p$ large enough.

The amount of smoothness of the solution $f$ depends on the decay of the initial data $f_0$. If the initial data $f_0$ belongs to $L^\infty_q$ for every exponent $q>0$, then the solution $f$ is $C^\infty$. If $f_0 \in L^\infty_q$ for some large exponent $q$, then the solution $f$ will be $C^k$ differentiable for some $k$ depending on $q$. This is simply from the regularity results for the non-cutoff Boltzmann equation with soft potentials that are already available in the literature (see \cite{imbert2020EMS,imbert2022,henderson2022classical}).

We prove Theorem \ref{t:global_existence} in Section \ref{s:global_existence} as a simple consequence of Theorem \ref{t:main} together with some results in the literature.

\subsection{Integro-differential Log-Sobolev inequalities on the sphere}

As explained above, the proof of the monotonicity of the Fisher information for the space-homogeneous Boltzmann equation consists first in reducing the problem to the proof of a functional inequality on the sphere, namely \eqref{e:spherical-inequality} (Theorem~\ref{t:main}), and then checking that this inequality holds true in various scenarios (Subsection \ref{subsec:comments}). To accomplish this program, we shall combine three inequalities: the first one relates the iterated carré du champ $\Gamma^2_{\B,\Delta}$ to the Fisher information (see \eqref{e:DtoFisher}), the second one is of Poincaré-type  (see \eqref{e:H1toNonlocal}) and the third one is an elementary inequality between nonlinear integral quantities (see \eqref{e:squareroots}).

The first of those inequalities, \eqref{e:DtoFisher}. is interesting in itself, and also implies, by a classical semigroup argument \`a la Bakry--\'Emery, the following new inequality (see Definition~\ref{d:bfisher} and Proposition~\ref{p:integral-log-sobolev}):
\begin{equation}\label{ineqDS}   
\iint_{S^{d-1} \times S^{d-1}} (f'-f) \log(f'/f) b(\sigma' \cdot \sigma) \dd \sigma ' \dd \sigma \geq 4 \cOne \left( \int_{S^{d-1}} f \log f \dd \sigma - \left( \int_{S^{d-1}} f \dd \sigma \right) \log \left( \fint_{S^{d-1}} f \dd \sigma \right) \right).
\end{equation}
In the diffusive limit when $b$ concentrates on grazing collisions, inequality~\eqref{ineqDS} reduces to the classical log-Sobolev inequality on the sphere \cite{bakry-2014}. But on the other hand, at the other extreme when $b$ is the constant kernel, \eqref{ineqDS} becomes the spherical version of the entropy-entropy production inequality established for the Boltzmann equation with super-hard spheres by the third author \cite{vill:cer:03}, following a previous joint work with Toscani \cite{TV:entropy:99}. So \eqref{ineqDS} bridges these two families of entropic functional inequalities coming from the study of heat equation on the one hand, and Boltzmann equation on the other hand.

\subsection{Review of literature}

We do not aim at exhaustivity here, and refer to the lecture notes by the third author \cite{villani-crete} for a broader perspective, historical notes and many other references on Fisher information and kinetic theory. Below is a selection of some of the works which are most directly relevant to the present study.

\subsubsection{Fisher information and kinetic theory}

H.~P.~McKean \cite{mckean1966} introduced Fisher information in kinetic theory with a view of understanding the convergence to equilibrium. He proved that the Fisher information is nonincreasing along solutions of the Kac equation, a one-dimensional caricature of the spatially homogeneous Boltzmann equation with Maxwell molecules. He also established the very first result of exponential convergence for an equation of Boltzmann type. A quarter of century later, G.~Toscani \cite{toscani1992} extended McKean's monotonicity result to the case of the space-homogeneous Boltzmann equation with Maxwell molecules in 2D. Then the third author \cite{villani1998boltzmann} extended Toscani's result to any dimension, still for Maxwell molecules. In \cite{villani2000landau}, he provided a direct proof for the space-homogeneous Landau equation, again with Maxwell molecules. Then the subject remained still for nearly another quarter of century, and many experts believed that this monotonicity was a peculiarity of Maxwell molecules. True, there was some numerical evidence indicating that the decrease would hold in more generality, but it was generally considered as non-conclusive. However, recently, N.~Guillen and the second author \cite{landaufisher2023} proved that the Fisher information decreases for the Landau equation for a general family of interaction potentials including the Coulomb potential in 3D. Our technique of reduction to a nonlinear integro-differential inequality on the sphere originates from that work.

\subsubsection{Log-Sobolev inequalities}

The log-Sobolev inequality was first established for the Gaussian measure by A.J.~Stam \cite{stam:59} in the context of information theory, and rediscovered independently by L.~Gross \cite{zbMATH03498656} for the purpose of infinite-dimensional Sobolev analysis.
 Then D.~Bakry and M.~\'Emery proposed a robust proof based on the study of the carré du champ operator $\Gamma$ and the iterated carré du champ $\Gamma^2$ \cite{zbMATH03894218}. Their analysis based on curvature-dimension bounds opened the path to a rich field with a number of ramifications in statistical physics, Riemannian geometry, and more generally the analysis of diffusion processes \cite{bakry-2014, vill:oldnew}.

N.~Guillen and the second author proved  in \cite{landaufisher2023} that the $\Gamma^2$ criterion appearing in \cite[Proposition~5.7.3]{bakry-2014} holds true on the projective plane $\mathbb R\mathbb P^2$ with an explicit constant -- see \eqref{e:be}. This estimate plays a key role in their proof of the monotonicity of the Fisher information for the Landau equation, and there was need to get a good estimate on the constant, applying known results was not sufficient. The equation \eqref{e:spherical-inequality} can be understood as an integro-differential version of this $\Gamma_2$ criterion on $\mathbb R\mathbb P^{d-1}$. The constant obtained in \cite{landaufisher2023} for the three dimensional case was recently improved and extended to any dimension by S.~Ji  \cite{sehyun2024}.

The study of log-Sobolev inequalities in  integro-differential forms is less advanced. The reader is referred first to \cite{zbMATH01562918,cm02} and then to \cite{chafai-2004} and the numerous references therein.

\subsubsection{Existence of global solutions for the space-homogeneous Boltzmann equation}

The study of existence of classical solutions for the space-homogeneous Boltzmann equation in the cutoff case is a classical subject going back to the work of Carleman \cite{carleman}. In the noncutoff case, the most complete result currently in the literature is by L.~He in \cite{he2012} when $\gamma+2s \geq 0$. In the same work, for the case of very soft potentials $\gamma+2s < 0$, the existence of solutions is presented only in a small interval of time. The existence of global strong solutions for very soft potentials (at least in the range $-2<\gamma+2s < 0$ where integral in time weak solutions can be constructed under only finite entropy and moment assumptions), has been a rather well known open problem in the community of kinetic theory for quite some time. The first weak solutions for this regime were constructed by the third author using entropy production estimates \cite{zbMATH01221506} at the end of the nineties; these ``$H$-solutions'' exist for all times and can be used for some qualitative results, see in particular the proof of convergence to equilibrium by Carlen--Carvalho--Lu \cite{carlen2009}. But it is not clear how regular or irregular they are, neither whether a uniqueness result can be obtained. This motivated a number of studies in the past decade. In particular, the second author obtained a conditional $L^\infty$ estimate in \cite{silvestre-2016}. There is an estimate in weighted $L^p$ spaces for the entropy production by J.~Chaker and the second author in \cite{zbMATH07705759}. A partial regularity result was obtained by F.~Golse and the first and second authors  \cite{golse2023partialregularitytimespacehomogeneous}. C.~Henderson, S.~Snelson and A.~Tarfulea established short-time existence and continuation criterion in \cite{henderson2022classical} and \cite{henderson2023decay} that apply also in the space inhomogeneous case. All those results were partial or conditional: we finally solve this quest, at least for all classical physically relevant kernels, in our Theorem \ref{t:global_existence}, thanks to the new Fisher information a priori estimate. As for the strongly related problem of global smoothness for the Landau equation with Coulomb potential, it was addressed similarly by N.~Guillen and the second author in \cite{landaufisher2023}.

\subsection{Organization of the article}

Section~\ref{s:prelim} mostly contains known results that will be used later on. It also contains a description of the type of solutions we consider in this paper, and a proof of the key fact that the operators $\Delta$ and $B$ commute.
In Section~\ref{s:lifting}, we introduce two useful tools from  \cite{landaufisher2023} for the analysis of the time evolution of the Fisher information: a lifting from $\R^d$ to $\R^{2d}$ and polar coordinates. In Section~\ref{s:carré-du-champ}, we introduce (iterated) carré-du-champ operators associated with the Laplacian on the sphere and the operator $\B$ defined in \eqref{e:spherical-operator}. 
Section~\ref{s:generic} is devoted to the proof of Theorem~\ref{t:main}.  In Section~\ref{s:curvature}, we prove that \eqref{e:spherical-inequality} holds true if $d \ge 3$ and $b$ satisfies \eqref{e:generic-b}. In Section~\ref{s:counterexample}, we exhibit a kernel $b$ for which \eqref{e:spherical-inequality} only holds true with $\Lambda_b=0$. We then turn our attention in Section~\ref{s:subordinate} towards kernels that are obtained by integrating the heat kernel against a weight and prove that \eqref{e:spherical-inequality} holds true with $\Lambda \geq d$ in that case. In the final Section~\ref{s:global_existence}, we explain how to get global existence of solutions for the space-homogeneous Boltzmann equation with very soft potentials. 
\bigskip

\paragraph{\bf Notation.}
The unit sphere in $\R^d$ is denoted by $S^{d-1}$. The scalar product in $\R^d$ between two vectors $x$ and $y$ is denoted by $x \cdot y$. For $v \in \R^3$, $\langle v \rangle$ denotes
$\sqrt{1+|v|^2}$. The norm of $f$ in $L^p_q (\R^d)$ is the same as the norm of $\langle v \rangle^q f(v)$ in $L^p(\R^d)$.

Functions $f,g \dots$ are defined on $\R^d$ while $F,G \dots$ are defined in $\R^{2d}$.
The Laplace-Beltrami operator on the sphere $S^{d-1}$ is denoted by $\Delta$. 
The operator $\B$ is defined in \eqref{e:spherical-operator}.
Carré du champ operators are denoted by $\Gamma_\B$ and $\Gamma_\Delta$ and the corresponding iterated one by $\Gamma^2_{\B,\Delta}$, see Section~\ref{s:carré-du-champ}. 
The finite dimensional linear operators $|\cdot|_{\sigma',\sigma}$, $P_{\sigma,\sigma'}$ and $M_{\sigma,\sigma'}$ are defined in \eqref{e:nonlocalconnection}, \eqref{e:def-P}  and \eqref{e:def-M}, respectively.

\section{Preliminaries}
\label{s:prelim}

\subsection{Short time existence and continuation criterion}

The best short-time existence result and continuation criterion currently available for the Boltzmann equation in the very soft potential range is due to Chris Henderson, Stanley Snelson, and Andrei Tarfulea in \cite{henderson2022classical} and \cite{henderson2023decay}. In the context of the homogeneous Boltzmann equation with very soft potentials, their results can be summarized in the following theorem.
Its statement involves the following norm,
\[ \|f \|_{L^\infty_q (\R^3)} = \| \langle v \rangle^q  f \|_{L^\infty (\R^3)}.\]
\begin{thm}[Short-time existence and continuation criterion -- \cite{henderson2022classical,henderson2023decay}] \label{t:hst}
  Assume $d=3$ and that the collision kernel $B$ has the form \eqref{e:non-cutoff-standard}. 
  If $f_0 \in L^\infty_q(\R^3)$ for some $q> 3+2s$, then there exists a solution $f$ to the Boltzmann equation \eqref{e:boltzmann} in $[0,T] \times \R^3$ for some $T>0$.

  This solution $f$ is strictly positive, and for any exponents $k,p \geq 0$, there exists $q = q(k,p)$ so that if $f_0 \in L^\infty_q (\R^3)$, then for any $\tau>0$,
\[ D^k f \in L^\infty([\tau,T],L^\infty_p(\R^3)).\]

Moreover, if such a solution exists for all $T < T_{\max}$, but it does not exist for any $T > T_{max}$, then necessarily
\[ \lim_{t \to T_{max}} \int_{\R^3} \langle v \rangle^{k_1} f^{p_1} \dd v = +\infty \]
Here, $k_1 = 1 - 3(\gamma+2s)/2s$ and $p_1$ is any exponent such that $p_1 > 3/(3+\gamma+2s)$.
\end{thm}

\begin{remark}
The results in \cite{henderson2022classical,henderson2023decay} apply to space-inhomogeneous equations as well. Much of the difficulty of these papers is due to this. There is an assumption in the main result of these papers about a lower bound for $f_0$ in certain balls that is only relevant in the space inhomogeneous setting.
\end{remark}

\subsection{Smooth decaying solutions}

We recall that the Fisher information is finite for smooth and decaying functions. The result is stated in terms of the following weight Sobolev norms: for an integer $k \ge 1$ and a positive exponent $q>0$,
\begin{align*}
  \|f \|_{H^k_q}^2 & = \sum_{|\alpha| \le k} \| \langle v \rangle^{q} D^\alpha f\|_{L^2}^2, \\
  \|f \|_{\dot H^k_q}^2 & = \sum_{|\alpha| = k} \| \langle v \rangle^{q} D^\alpha f\|_{L^2}^2.
\end{align*}
\begin{lemma}[Smooth decaying functions have finite Fisher information -- \cite{zbMATH01486430}] \label{l:finite-fisher}
  For all $\eps >0$, there exists $C_\eps >0$ (only depending on dimension and $\eps$) such that,
  \[ I (f) \le C_\eps  \left( \|f \|_{\dot H^2_{(d/2)+\eps}} + \|f \|_{\dot H^1_{(d/2)-1+\eps}} + \|f \|_{L^2_{(d/2)-2+\eps}} \right) \le C_\eps \| f\|_{H^2_{(d/2)+\eps}}.\]
\end{lemma}

Theorem~\ref{t:hst} ensures that there exists a strictly positive classical solution $f$ of \eqref{e:boltzmann} at least in a short interval of time $[0,T]$. Assuming that $f_0$ has an appropriate decay, the solution satisfies $\langle v \rangle^q f(t,v) \in L^\infty$, for some large exponent $q$. Moreover, for any $\tau > 0$, 
\[ D^2 f \in L^\infty([\tau,T],L^\infty_q(\R^3)).\]
Lemma \ref{l:finite-fisher} above ensures that these solutions have a finite Fisher information for any $t>0$. This is the only notion of solution we need to consider in this paper. In Theorem \ref{t:global_existence} we prove that these solutions can be extended globally in time without singularities. Thus, weak solutions are no longer necessary in this context.

\subsection{Commutation of $\Delta$ and $\B$}

\begin{lemma}[The operators $\B$ and $\Delta$ commute] \label{l:BandDelta_commute}
For any smooth function $f:S^{d-1}\to \R$, $\Delta \B f = \B \Delta f$. 
\end{lemma}

\begin{proof}
Let us assume without loss of generality that both $f$ and $b$ are smooth. The general case follows by a standard approximation argument.

We compute $\Delta \B f$ by differentiating inside the integral.
\begin{align*}
\Delta \B f &= \Delta \int_{S^{d-1}} (f(\sigma')-f(\sigma)) b(\sigma'\cdot\sigma) \dd \sigma' \\
&= \int_{S^{d-1}} -\Delta f(\sigma) b(\sigma'\cdot\sigma) + (f(\sigma')-f(\sigma)) \Delta_\sigma[b(\sigma'\cdot\sigma)] \dd \sigma'.
\end{align*}
A practical way to compute the spherical Laplacian of a function is to extend the function to $\R^d$ and use the formula for the Laplacian in polar coordinates. With this in mind, we write $u(x) = b(\sigma'\cdot x)$ for $x = r\sigma$ and compute
\begin{align*}
  \Delta_x [b(\sigma'\cdot x)] &= \partial_{rr} [b(\sigma'\cdot x)] + \frac{d-1}r \partial_r [b(\sigma'\cdot x)] + \frac 1 {r^2} \Delta_\sigma [b(\sigma'\cdot x)]. \\
  \intertext{We can compute the Laplacian in Euclidian coordinates in the left hand side and partial derivatives w.r.t. $r$ in the right hand side and get,}
b''(\sigma' \cdot x) &= b''(\sigma' \cdot x) (\sigma' \cdot \sigma)^2 + \frac{d-1}r (\sigma' \cdot \sigma) b'(\sigma'\cdot x) + \frac 1 {r^2} \Delta_\sigma [b(\sigma'\cdot x)].
\end{align*}
Evaluating at $r=1$, we obtain,
\begin{align*}
\Delta_\sigma [b(\sigma'\cdot \sigma)] &= \left( 1 - (\sigma' \cdot \sigma)^2 \right) b''(\sigma' \cdot \sigma) - (d-1) (\sigma' \cdot \sigma) b'(\sigma' \cdot \sigma). 
\end{align*}
The only thing that matters for this proof is that $\sigma$ and $\sigma'$ are exchangeable in the formula above. It means that $\Delta_\sigma [b(\sigma'\cdot \sigma)] = \Delta_{\sigma'} [b(\sigma'\cdot \sigma)]$. Introducing this fact in our computation above for $\Delta \B f$, and then integrating by parts, we get
\begin{align*}
\Delta \B f &= \int_{S^{d-1}} -\Delta f(\sigma) b(\sigma'\cdot\sigma) + (f(\sigma')-f(\sigma)) \Delta_{\sigma'}[b(\sigma'\cdot\sigma)] \dd \sigma' \\
&= \int_{S^{d-1}} -\Delta f(\sigma) b(\sigma'\cdot\sigma) +  \Delta f(\sigma') b(\sigma'\cdot\sigma) \dd \sigma' \\
&= \B \Delta f. \qedhere
\end{align*}
\end{proof}

\subsection{Comparison with reference kernels}

We conclude this preliminary section with the precise statement of a remark contained in the introduction.

\begin{prop}[Comparison with reference kernels]\label{p:perturbation}
Assume that $b_0$ is a reference kernel for which \eqref{e:spherical-inequality} holds with parameter $\Lambda_0$. Let $b$ be another kernel such that \eqref{e:kernel-comparison} holds true. 
Then \eqref{e:spherical-inequality} holds for $b$ with $\Lambda_b = c_1 \Lambda_0/C_2.$
\end{prop}
\begin{proof}
  Let $\B_0$ denote the linear integro-differential operator associated with the reference kernel $b_0$.
  Thanks to Lemma~\ref{l:Gamma2-with-T}, we know  that $\Gamma^2_{\B,\Delta} \ge c_1 \Gamma^2_{\B_0,\Delta}$. It is now easy to conclude.
\end{proof}

\section{Lifting and polar coordinates}
\label{s:lifting}

\subsection{Lifting the problem}

Following \cite{landaufisher2023}, we lift the problem to the double variables $(v,w) \in \R^{d+d}$. For a function $F : \R^{2d} \to [0,\infty)$, we set
\[ I(F) := \iint_{\R^{2d}} \frac{|\nabla F|^2} F \dd w \dd v.\]
For a function $F \colon \R^{2d} \to \R$, we can consider the marginals with respect to the two variables $v$ and $w$.
If $F(v,w) = F(w,v)$, these marginals coincide. We write $\pi F$ to denote this marginal
\[ \pi F (v) = \int_{\R^d} F(v,w) \dd w.\]
A classical property of the Fisher information is that it is larger than the sum of the Fisher information of marginals.
If $F$ is symmetric, this means that
\begin{equation}
  \label{e:fisher-margin}
  I (F) \ge 2 i (\pi F).
\end{equation}

We also define $Q(F)$ so that $q(f) = \pi Q(f \otimes f)$,
\begin{equation}  \label{e:Q}
Q(F) := \int_{\sigma' \in \R^{d-1}} (F(v',w') - F(v,w)) B(|v-w|,\cos \theta ) \dd \sigma'.
\end{equation}
Note that $Q(F)$ is a linear elliptic integro-differential operator whose kernel is supported on $(d-1)$-dimensional spheres.

We write the integral in $Q$ with respect to the variable $\sigma' \in S^{d-1}$ because we will soon write a change of variables $v = z + r \sigma$ and $w = z - r \sigma$. We want to differentiate between the $\sigma \in S^{d-1}$ from $(v,w)$ after the change of variables, and the $\sigma'$ of the integral in the expression for $Q$ so that $v' = z + r \sigma'$ and $w' = z - r \sigma'$.

Because of \eqref{e:fisher-margin}, it is explained in \cite[Lemma~3.4 and Remark 3.5]{landaufisher2023}  that we have for $F = f \otimes f$,
\[ \langle i'(f),q(f) \rangle = \frac 12 \langle I'(F), Q(F) \rangle.\]

In order to prove Theorem \ref{t:main}, it is thus sufficient to prove that the following inequality holds for every smooth function $F \colon \R^{2d} \to [0,\infty)$ that is symmetric $F(v,w) = F(w,v)$,
\begin{equation} \label{e:aim}
\langle I'(F), Q(F) \rangle \leq 0.
\end{equation}

\subsection{Polar coordinates}

It is useful to consider the following change of variables,
\begin{align*}
v = z + r \sigma, \\
w = z - r \sigma.
\end{align*}
Here $z \in \R^d$, $\sigma \in S^{d-1}$ and $r \in [0,\infty)$.

We can write the Fisher information with respect to these variables,
\begin{eqnarray}\nonumber I(F) &=& \iiint \left( \frac{|\nabla_z F|^2} F + \frac{|\partial_r F|^2} F + \frac{|\nabla_\sigma F|^2} {r^2 F} \right) r^{d-1} \dd \sigma \dd r \dd z, \\
\label{e:fisher-split} &=:& I_{parallel} + I_{radial} + I_{spherical}.
\end{eqnarray}
For the sake of clarity, we abuse notation from now on by writing $F$ both for the function defined in $\R^{2d}$ and the one defined in $\R^d \times (0,+\infty) \times S^{d-1}$ in terms of polar coordinates.

We also write $Q(F)$ in terms of the new variables and realize that it takes a very simple form. In fact, it only acts on the spherical variable $\sigma$,
\[ Q(F) = \alpha(r) \int_{\sigma' \in S^{d-1}} (F(z,r,\sigma') - F(z,r,\sigma)) b(\sigma' \cdot \sigma) \dd \sigma'. \]

The operator $Q$ consists in applying the spherical diffusion operator $\B$ to the function $F(z,r,\cdot)$, for every fixed value of $z \in \R^d$ and $r>0$, and then multiplying it by the scalar function $\alpha(r)$.

To prove Theorem \ref{t:main}, we differentiate each term $I_{parallel}$, $I_{radial}$ and $I_{spherical}$ in the direction of $Q(F)$.

\begin{lemma}[Differentiating $I_{parallel}$] \label{l:Iparallel}
  The following identity holds,
\[ \langle I'_{parallel}(F), Q(F) \rangle = \frac 12 \iiiint -\alpha \left| \nabla_z \log F' - \nabla_z \log F \right|^2 (F'+F) \ b(\sigma'\cdot\sigma) r^{d-1} \dd \sigma' \dd \sigma \dd r \dd z.\]
Here, we write $F = F(z,r,\sigma)$ and $F'=F(z,r,\sigma')$.
\end{lemma}

\begin{proof}
We start from the integral defining $I_{parallel}$ and differentiate in the direction $Q$. Using $\nabla_z \alpha=0$, we symmetrize the integral.
\begin{align*}
\langle &I'_{parallel}(F), Q(F) \rangle = \iiint \left(2 \nabla_z \log F \cdot \nabla_z Q - |\nabla_z \log F|^2 Q \right) \ r^{d-1} \dd \sigma \dd r \dd z \\
&= \iiiint \alpha \big(2 \nabla_z \log F \cdot (\nabla_z F'-\nabla_z F) - |\nabla_z \log F|^2 (F'-F) \big) \ b(\sigma'\cdot\sigma) r^{d-1} \dd \sigma' \dd \sigma \dd r \dd z, \\
&= -\frac 12 \iiiint \alpha \Big( 2(\nabla_z \log F' - \nabla_z \log F) \cdot (\nabla_z F'-\nabla_z F) \\
&\qquad \qquad - (|\nabla_z \log F'|^2 - |\nabla_z \log F|^2) (F'-F) \Big) \ b(\sigma'\cdot\sigma) r^{d-1} \dd \sigma' \dd \sigma \dd r \dd z \\
&= -\frac 12 \iiiint \alpha (\nabla_z \log F' - \nabla_z \log F) \cdot \Big( 2 (\nabla_z F'-\nabla_z F) \\
& \qquad \qquad - (\nabla_z \log F' + \nabla_z \log F) (F'-F) \Big) \ b(\sigma'\cdot\sigma) r^{d-1} \dd \sigma' \dd \sigma \dd r \dd z \\
&= -\frac 12 \iiiint \alpha |\nabla_z \log F' - \nabla_z \log F|^2 (F'+F) \ b(\sigma'\cdot\sigma) r^{d-1} \dd \sigma' \dd \sigma \dd r \dd z. \qedhere
\end{align*}
\end{proof}

\begin{lemma}[Differentiating $I_{radial}$] \label{l:Iradial}
The following inequaliy holds,
\[
\langle I'_{radial}(F), Q(F) \rangle \leq \frac 12 \iiiint \frac{(\dalpha)^2}{\alpha} \frac{(F' - F)^2} {F'+F} \ b(\sigma'\cdot\sigma) r^{d-1} \dd \sigma' \dd \sigma \dd r \dd z.
\]
Here, we write $F = F(z,r,\sigma)$ and $F'=F(z,r,\sigma')$.
\end{lemma}

\begin{proof}
We write down the integral and symmetrize as in the previous lemma.
\begin{align*}
\langle &I'_{radial}(F), Q(F) \rangle = \iiint \left(2 \partial_r \log F \cdot \partial_r Q - |\partial_r \log F|^2 Q \right) \ r^{d-1} \dd \sigma \dd r \dd z, \\
\intertext{note that the integrand in $Q$ includes a factor $\alpha$ that depends on $r$,}
&= \iiiint \alpha(r) \big(2 \partial_r \log F \cdot (\partial_r F'- \partial_r F) - |\partial_r \log F|^2 (F'-F) \big) \ b(\sigma'\cdot\sigma) r^{d-1} \dd \sigma' \dd \sigma \dd r \dd z \\
& \qquad + \iiiint 2 \dalpha(r) \ \partial_r \log F  (F'- F) \ b(\sigma'\cdot\sigma) r^{d-1} \dd \sigma' \dd \sigma \dd r \dd z, \\
\intertext{we symmetrize in $\sigma$ and $\sigma'$,}
&= -\frac 12 \iiiint \alpha |\partial_r \log F' - \partial_r \log F|^2 (F'+F) \ b(\sigma'\cdot\sigma) r^{d-1} \dd \sigma' \dd \sigma \dd r \dd z \\
& \qquad + \iiiint \dalpha(r) \ (\partial_r \log F - \partial_r \log F')  (F'- F) \ b(\sigma'\cdot\sigma) r^{d-1} \dd \sigma' \dd \sigma \dd r \dd z, \\
\intertext{we complete squares,}
&\leq \iiiint \frac{\dalpha(r)^2}{2\alpha(r)} \frac{(F'- F)^2}{F'+F} \ b(\sigma'\cdot\sigma) r^{d-1} \dd \sigma' \dd \sigma \dd r \dd z. \qedhere
\end{align*}
\end{proof}

We are left with differentiating the spherical part of the Fisher information. It is a more delicate computation that we explain in a later section. The case $d=2$ is much simpler and we explain it right here.

\begin{lemma}[Differentiating $I_{spherical}$ in 2D] \label{l:Iradial-2D}
Let us assume that $d=2$. In this case we can write the function $F = F(z,r,\theta)$ where $\theta$ is an angular variable and $F$ is $\pi$-periodic with respect to $\theta$.

The following identity holds
\[ \langle I'_{spherical}(F), Q(F) \rangle = \frac 12 \iiiint -\alpha \left| \partial_\theta \log F' - \partial_\theta \log F \right|^2 (F'+F) \ b(\cos(\theta'-\theta)) r^{d-3} \dd \theta' \dd \theta \dd r \dd z.\]
Here, we write $F = F(z,r,\theta)$ and $F'=F(z,r,\theta')$.
\end{lemma}

\begin{proof}
We write down the integral and symmetrize as in Lemma \ref{l:Iparallel}.
\end{proof}

\section{Carré du champ operators on the sphere}
\label{s:carré-du-champ}

In this section we analyze the \emph{carré du champ} operators and the corresponding Bakry-\'Emery formalism on the sphere corresponding to the heat flow and the flow by the integro-differential diffusion $\B$.

\subsection{Definitions and first properties}

The \emph{carré du champ} operator associated to the Laplace-Beltrami operator on the sphere $S^{d-1}$ consists in the following formula,
\begin{equation} \label{e:silly-carré-du-champ}
\Gamma_\Delta(f,g) := \frac 12 \left( \Delta(f g) - g \Delta f- f \Delta g \right).
\end{equation}
A direct computation shows that $\Gamma_\Delta(f,g)$ is simply equal to $\nabla f \cdot \nabla g$.

We may consider the carré du champ operator associated to the operator $\B$ instead of the Laplacian, 
\begin{equation} \label{e:B-carré-du-champ}
\Gamma_\B(f,g) := \frac 12 \left( \B(f g) - g \B f- f \B g \right).
\end{equation}
From a direct computation, one can verify that 
\[ \Gamma_\B(f,g) = \frac 12 \int_{S^{d-1}} \left( f(\sigma')-f(\sigma) \right) \left( g(\sigma')-g(\sigma) \right) b(\sigma' \cdot \sigma) \dd \sigma'.\]

Some familiar looking integration-by-parts formulas hold with the operators $\Delta$, $\B$ and their corresponding carré du champ operators. We state them in the following lemma.
\begin{lemma}[Carré du champ operators and integration by parts]
Given any two functions $f,g: S^{d-1} \to \R$, the following identities hold.
\begin{align*}
\int_{S^{d-1}} \Gamma_\Delta(f,\B g) \dd \sigma = \int_{S^{d-1}} \Gamma_\Delta(\B f, g) \dd \sigma, \\
\int_{S^{d-1}} \Gamma_\B(f,\Delta g) \dd \sigma = \int_{S^{d-1}} \Gamma_\B(\Delta f, g) \dd \sigma.
\end{align*}
\end{lemma}

\begin{proof}
This is a simple consequence of the fact that $\Delta$ and $\B$ commute, and they are self-adjoint.
\end{proof}

For a function $f \colon S^{d-1} \to (0,\infty)$, we may write its usual entropy as $\hs(f)$ and its Fisher information $\is (f)$,
\[ \hs(f) := -\int_{S^{d-1}} f \log f \dd \sigma \quad \text{ and } \quad \is(f) := \int_{S^{d-1}} \frac{|\nabla f|^2} f \dd \sigma .\]
When we differentiate the entropy with respect to the heat flow, we obtain the usual Fisher information.
\begin{align*}
 \langle \hs'(f), \Delta f\rangle &= -\int_{S^{d-1}} \Delta f \cdot \log f \dd \sigma \\
                                &= \int_{S^{d-1}} \Gamma_\Delta(f, \log f) \dd \sigma \\
                                &= \int_{S^{d-1}} f \Gamma_\Delta(\log f, \log f) \dd \sigma \\
                                  &= \int_{S^{d-1}} \frac{|\nabla f|^2} f \dd \sigma \\
                                  & = \is (f).
\end{align*}

When we differentiate the entropy along the flow of $\B$ instead, we obtain a different integral quantity.
\begin{equation} \label{e:dHdB}
\begin{aligned}
\langle \hs'(f), \B f\rangle &= -\int_{S^{d-1}} \B f \cdot \log f \dd \sigma \\
&= \int_{S^{d-1}} \Gamma_\B(f, \log f) \dd \sigma \\
&= \frac 12 \iint_{S^{d-1} \times S^{d-1}} (f(\sigma') - f(\sigma)) \log \left( \frac{f(\sigma')}{f(\sigma)} \right) b(\sigma'\cdot \sigma) \dd \sigma' \dd \sigma.
\end{aligned}
\end{equation}

It is easy to check that, due to the monotonicity of the $\log$ function, $\Gamma_\B(f,\log f) \geq 0$ pointwise. However, unlike the case of the Laplacian, in general we have $\Gamma_\B(f, \log f) \neq f \Gamma_\B(\log f, \log f)$.

We will now discuss the second order \emph{carré du champ} operators associated to $\Delta$ and $\B$. The usual second order carré du champ is given by
\[ \Gamma_2(f,g) = \Gamma^2_{\Delta,\Delta}(f,g) = \frac 12 \left( \Delta \Gamma(f,g) - \Gamma(\Delta f,g) - \Delta \Gamma(f,\Delta g) \right).\]
The second order carré du champ operators that we want to consider mix the flow of $\Delta$ and $\B$. We define
\begin{equation} \label{e:2nd-order-carré-du-champ}
\begin{cases}
\Gamma^2_{\Delta,\B}(f,g) &:= \frac 12 \left( \Delta \Gamma_\B(f,g) - \Gamma_\B(\Delta f,g) - \Gamma_\B(f,\Delta g)\right), \\[1ex]
\Gamma^2_{\B,\Delta}(f,g) &:= \frac 12 \left( \B \Gamma_\Delta(f,g) - \Gamma_\Delta(\B f,g) - \Gamma_\Delta(f,\B g)\right), \\[1ex]
\Gamma^2_{\B,\B}(f,g) &:= \frac 12 \left( \B \Gamma_\B(f,g) - \Gamma_\B(\B f,g) - \Gamma_\B(f,\B g)\right).
\end{cases}
\end{equation}

Since the operators $\Delta$ and $\B$ commute, it is not difficult to verify that $\Gamma_{\Delta,\B}^2(f,g) = \Gamma_{\B,\Delta}^2(f,g)$. We include the definition of $\Gamma_{\B,\B}$ here for completeness only, since it is not used in this paper.

\subsection{Derivatives of the Fisher information along the flow of $\B$}

With these operators, we can write an expression for the derivative of the Fisher information along the flow of $\B$.
\begin{lemma}[Differentiating Fisher along the flow of $\B$] \label{l:fisher-along-B}
\[ \langle \is'(f), \B f \rangle = -2 \int_{S^{d-1}} f \Gamma^2_{\B,\Delta}(\log f, \log f) \dd \sigma.\]
\end{lemma}

\begin{proof}
We start from the identity for the integrand $|\nabla f|^2/f = f \Gamma_\Delta(\log f, \log f)$ and differentiate it along the direction $\B f$.
\begin{align*}
\langle \is'(f), \B f \rangle &= \int_{S^{d-1}} \B f \Gamma_\Delta(\log f, \log f) + 2 f \Gamma_\Delta(\log f, \B f /f) \dd \sigma \\
&= \int_{S^{d-1}} f \B \Gamma_\Delta(\log f, \log f) + 2 \nabla f \cdot \nabla \left( \frac {\B f} f \right) \dd \sigma \\
                            &= \int_{S^{d-1}} f \B \Gamma_\Delta(\log f, \log f) + 2 \nabla \log f \cdot \nabla \B f - 2 \B f \Gamma_\Delta(\log f, \log f) \dd \sigma \\
&= -2 \int_{S^{d-1}} f \Gamma_{\B,\Delta}^2(\log f, \log f) \dd \sigma. 
\end{align*}
To get the last line, we used the fact that $\B$ is self-adjoint together with the following equality,
  \[  \int \nabla \log f \cdot \nabla \B f \dd \sigma = \int \nabla \B \log f \cdot \nabla f \dd \sigma = \int f \Gamma_\Delta (\B \log f, \log f) \dd \sigma. \qedhere \]
\end{proof}

The integrand in Lemma \ref{l:fisher-along-B} is nonnegative for all $\sigma \in S^{d-1}$. To see that, we will derive some explicit expressions for $\Gamma^2_{\B,\Delta}$. It involves the quantity $\nabla \B f$. Naively, we would expect its expression to contain a factor $(\nabla f' - \nabla f)$ in the integrand. However, $\nabla f(\sigma')$ and $\nabla f(\sigma)$ belong to the tangent bundle of $S^{d-1}$ at different points. We cannot directly subtract them. We propose two methods to circunvent this difficulty and derive a correct expression for $\nabla \B f$. One is to differentiate along the vector fields used in \cite{landaufisher2023} for the Landau equation, the other is to use the maps from the tangent spaces at different points that are defined in \cite{villani1998boltzmann}. In this paper, we use the second approach, which is more natural in the context of the Boltzmann equation. The vector fields used in \cite{landaufisher2023} are more natural in the context of the Landau equation. In Appendix \ref{s:vector-fields}, we explain how the formulas involving $\Gamma^2_{\B,\Delta}$ can be written in terms of these vector fields.

\subsection{Comparing vectors in different tangent spaces}
\label{s:maps-between-tangent-spaces}

A possible approach to make computations involving  $\nabla \B f$ is to use the  linear transformations between $T_\sigma S^{d-1}$ and $T_{\sigma'} S^{d-1}$ described in \cite{villani1998boltzmann}. We describe the method here.

\subsubsection{The gradient of the linear integro-differential operator}

Following \cite{villani1998boltzmann}, we define the linear transformation $M_{\sigma',\sigma}: \R^d \to \R^d$ as
\begin{equation}
  \label{e:def-M}
  M_{\sigma',\sigma}(x) = (\sigma' \cdot \sigma) x - (\sigma \cdot x) \sigma'.
\end{equation}
Note that $M_{\sigma',\sigma}(x) \perp \sigma$ for any values of $\sigma, \sigma' \in S^{d-1}$ and $x \in \R^d$.

The following lemma is proved in \cite[Lemma 2]{villani1998boltzmann}. The proof is essentially to integrate by parts with respect to $\sigma'$, and work out the resulting formulas.

\begin{lemma} \label{l:cedric-lemma2}
Let $\sigma \in S^{d-1}$ be fixed, and $f: S^{d-1} \to \R$. Then
\[ \int_{S^{d-1}} f(\sigma') \Pi_{\sigma^\perp}(\sigma') \; \db(\sigma'\cdot\sigma) \dd \sigma' = \int_{S^{d-1}} M_{\sigma',\sigma}(\nabla f(\sigma')) \;b(\sigma' \cdot \sigma) \dd \sigma'. \]
Here $b'$ is the derivative of the function $b$, and $\Pi_{\sigma^\perp}(x) = x - (x \cdot \sigma) \sigma$ stands for the projection operator onto the tangent space $T_\sigma S^{d-1}$.
\end{lemma}

Lemma \ref{l:cedric-lemma2} is useful to provide an expression for $\nabla \B f(\sigma)$. This is expressed in the following lemma, which is similar to Lemma \ref{l:flowing-B-with-bi} but in terms of these new operators $M_{\sigma',\sigma}$.
\begin{lemma}[Gradient of the integro-differential operator] \label{l:gradBf}
Let $f:S^{d-1} \to \R$ be an arbitrary smooth function. The following expression holds for $\nabla \B f(\sigma)$.
\[ \nabla \B f(\sigma) = \int_{S^{d-1}} \left( M_{\sigma',\sigma} \nabla f(\sigma') - \nabla f(\sigma) \right) b(\sigma'\cdot \sigma) \dd \sigma'. \]
\end{lemma}

\begin{proof}
We start from the integral expression for $\B f$ and differentiate inside the integral sign.
\begin{align*}
\nabla \B f(\sigma) &= \nabla \int (f(\sigma') - f(\sigma)) b(\sigma' \cdot \sigma) \dd \sigma', \\
\intertext{we differentiate the integrand with respect to $\sigma$,}
&= \int -\nabla f(\sigma) b(\sigma' \cdot \sigma) + (f(\sigma') - f(\sigma)) \db(\sigma' \cdot \sigma) \Pi_{\sigma^\perp} (\sigma') \dd \sigma',
\intertext{we apply Lemma \ref{l:cedric-lemma2} to the second term,}
&= \int -\nabla f(\sigma) b(\sigma' \cdot \sigma) + M_{\sigma',\sigma} \nabla_{\sigma'}(f(\sigma') - f(\sigma)) b(\sigma' \cdot \sigma) \dd \sigma' \\
&= \int_{S^{d-1}} \left( M_{\sigma',\sigma} \nabla f(\sigma') - \nabla f(\sigma) \right) b(\sigma'\cdot \sigma) \dd \sigma'. \qedhere
\end{align*}
\end{proof}

\subsubsection{Formula for an iterated carré du champ}

We also define
\begin{equation}
  \label{e:def-P}
  P_{\sigma',\sigma}(x) := M_{\sigma',\sigma}(x) + (\sigma'\cdot x) \sigma = (\sigma' \cdot \sigma) x + (\sigma'\cdot x) \sigma - (\sigma \cdot x) \sigma'.
\end{equation}
Note that if $x \perp \sigma'$, then $P_{\sigma',\sigma}(x) = M_{\sigma',\sigma}(x)$, and in particular $P_{\sigma',\sigma}(x) \perp \sigma$. The operator $P_{\sigma',\sigma}(x)$ maps $T_{\sigma'} S^{d-1}$ into $T_{\sigma} S^{d-1}$.

We will only apply $P_{\sigma',\sigma}(x)$ to vectors in $T_{\sigma'} S^{d-1}$. In this subspace of $\R^d$, the operators $P_{\sigma',\sigma}(x)$ and $M_{\sigma',\sigma}(x)$ are indistinguishable. The operator $P_{\sigma',\sigma}(x)$ has a simple description in the full space $\R^d$ that is worth keeping in mind. When $\sigma = \sigma'$, it is clear that $P_{\sigma',\sigma}(x)$ is the identity. When $\sigma \neq \sigma'$, the operator $P_{\sigma',\sigma}(x)$ is a rotation on the two-dimensional subspace generated by $\sigma$ and $\sigma'$ that maps $\sigma'$ to $\sigma$. Moreover, $P_{\sigma',\sigma}(x)$ is $(\sigma'\cdot \sigma)$ times the identity on the orthogonal complement to $\sigma$ and $\sigma'$. The following lemma, whose proof can be found in \cite[Lemma 4]{villani1998boltzmann}, reflects this characterization of $P_{\sigma,\sigma'}$.

\begin{lemma} \label{l:Pcontracts}
For any vector $x \in \R^d$, we have $|P_{\sigma,\sigma'} (x)| \leq |x|$.

Moreover, if $\Pi$ is the orthogonal projection of $\R^d$ onto the subspace generated by $\sigma$ and $\sigma'$, then
\[ |P_{\sigma',\sigma} (x)|^2 = (\sigma' \cdot \sigma)^2 |\Pi_\perp x|^2 + |\Pi x|^2.\]
\end{lemma}

The operator $P_{\sigma',\sigma}$ is not self-adjoint. Its adjoint coincides with switching the indices $\sigma'$ and $\sigma$.
\[ P_{\sigma',\sigma}^\ast(x) = (\sigma' \cdot \sigma) x - (\sigma'\cdot x) \sigma + (\sigma \cdot x) \sigma' = P_{\sigma,\sigma'}.\]

For a vector $x$ tangent to $S^{d-1}$ at $\sigma'$ and $y$ tangent to $S^{d-1}$ at $\sigma$, we define
\begin{equation} \label{e:nonlocalconnection}
\begin{aligned}
|y-x|_{\sigma',\sigma}^2 &:= |x|^2 + |y|^2 - 2 x \cdot P_{\sigma',\sigma}(y) \\
&= |x|^2 + |y|^2 - 2 P_{\sigma,\sigma'} (x) \cdot y.
\end{aligned}
\end{equation}
We stress that $|y-x|_{\sigma',\sigma}^2$ is defined by abuse of notation. The vectors $x$ and $y$ belong to tangent spaces at different points. The difference $y-x$ inside the bars does not make sense by itself.

Using the notation above, we provide an expression for $\Gamma^2_{\B,\Delta}$.

\begin{lemma}[Formula for an iterated carré du champ] \label{l:Gamma2-with-T}
For any smooth function $f : S^{d-1} \to \R$, the following expression holds,
\[ \Gamma^2_{\B,\Delta}(f,f) = \frac 12 \int_{S^{d-1}} |\nabla f(\sigma') - \nabla f(\sigma)|_{\sigma',\sigma}^2 b (\sigma' \cdot \sigma) \dd \sigma'. \]
In particular, $\Gamma^2_{\B,\Delta}(f,f) \geq 0$.

Here, where we write $|y-x|_{\sigma',\sigma}$ we mean the expression defined in \eqref{e:nonlocalconnection}.
\end{lemma}

\begin{proof}
By the definition of $\Gamma^2_{\B,\Delta}(f,f)$ we have
\begin{align*}
\Gamma^2_{\B,\Delta}(f,f) &= \frac 12 \B |\nabla f|^2 - \nabla f \cdot \nabla \B f, \\
\intertext{we use Lemma \ref{l:gradBf} for the second term,}
&= \int \left(\frac 12|\nabla f'|^2 - \frac 12|\nabla f|^2 - 2 \nabla f \cdot \left( M_{\sigma',\sigma} \nabla f' - \nabla f \right) \right) b(\sigma'\cdot \sigma) \dd \sigma' \\
&= \frac 12 \int \left(|\nabla f'|^2 + |\nabla f|^2 - 2 \nabla f \cdot M_{\sigma',\sigma} \nabla f' \right) b(\sigma' \cdot \sigma) \dd \sigma' \\
\intertext{we note that $\nabla f' \perp \sigma'$. In particular, $M_{\sigma',\sigma} \nabla f' = P_{\sigma',\sigma} \nabla f'$,}
&= \frac 12 \int \left(|\nabla f'|^2 + |\nabla f|^2 - 2 \nabla f \cdot P_{\sigma',\sigma} \nabla f' \right) b(\sigma' \cdot \sigma) \dd \sigma' \\
&= \frac 12 \int |\nabla f'-\nabla f|_{\sigma',\sigma}^2 b(\sigma' \cdot \sigma) \dd \sigma'. \qedhere
\end{align*}
\end{proof}

\subsection{The entropy production by integro-differential diffusion.}
\label{s:integral-log-sobolev}

In the previous section, we related the left hand side of \eqref{e:spherical-inequality} with the second order carré du champ operator $\Gamma^2_{\B,\Delta}$ and to the second derivative of the entropy in the directions of $\Delta f$ and $\B f$. It is natural to wonder if there is a similar characterization for the right hand side of \eqref{e:spherical-inequality} as some sort of integro-differential Fisher information. In this section, we discuss that question and show some relevant inequalities. 

In analogy with the classical Fisher information, let us name the quantity in \eqref{e:dHdB} as $-\IB(f)$ .
\begin{defn}[Entropy production for $\B$] \label{d:bfisher}
  Given a kernel function $b \colon [-1,1] \to [0,+\infty)$, the \emph{integro-differential Fisher information} is defined as,
\[ \IB (f) = \frac 12 \iint_{S^{d-1} \times S^{d-1}} (f'-f) \log(f'/f) b(\sigma' \cdot \sigma) \dd \sigma ' \dd \sigma.\]
\end{defn}
We recall the two other useful forms of $\IB$,
\begin{align*} \IB(f) &= \int_{S^{d-1}} \Gamma_\B(f, \log f) \dd \sigma \\
    &= \langle \hs'(f), \B f\rangle.
\end{align*}

The quantity $\IB$ is the derivative of the entropy by the flow of $\B$, in the same way that the Fisher information is the derivative of the entropy by the heat flow. In that sense, the quantity $\B$ might be interpreted as some integro-differential form of the Fisher information.

The following proposition, is similar to Lemma \ref{l:fisher-along-B} but with the opposite order of the operators.
\begin{lemma}[Differentiating the integro-differential Fisher along the heat flow] \label{l:BFisher-along-Delta}
\[ \langle \IB'(f), \Delta f \rangle = -2 \int_{S^{d-1}} f \Gamma^2_{\Delta,\B}(\log f, \log f) \dd \sigma.\]
\end{lemma}

\begin{proof}
Recall that $\IB(f) = \langle \hs'(f), \B f \rangle$. The quantity $\langle \IB'(f), \Delta f \rangle$ corresponds to differentiating $\hs(f)$ first in the direction of $\B$ and then in the direction of $\Delta$. Thus,
\begin{align*}
\langle \IB'(f), \Delta f \rangle &= \langle \hs''(f) \B f , \Delta f \rangle + \langle \hs'(f) , \B \Delta f \rangle.\\
\intertext{Observe that since $\hs''$ is symmetric and $\B \Delta f = \Delta \B f$, we obtain the same expression as if we differentiate in the direction of $\B f$ and $\Delta f$ in the opposite order.}
\langle \IB'(f), \Delta f \rangle &= \langle \is'(f), \B f \rangle \\
&= -2 \int_{S^{d-1}} f \Gamma^2_{\Delta,\B}(\log f, \log f) \dd \sigma. 
\end{align*}
We applied Lemma \ref{l:fisher-along-B} to get the last line. 
\end{proof}

When $a,b$ are two arbitrary positive numbers, the following elementary inequalities hold.
\[ (a-b)\log(a/b) \geq 4 (\sqrt a - \sqrt b)^2 \geq 2\frac{(a-b)^2}{a+b}.\]
Therefore, the following chain of elementary inequalities holds starting from the right-hand side of \eqref{e:spherical-inequality}.

\begin{align*}
\iint_{S^{d-1} \times S^{d-1}} \frac{(f(\sigma') - f(\sigma))^2} {f(\sigma')+f(\sigma)} b(\sigma' \cdot \sigma) \dd \sigma ' \dd \sigma &\leq 2 \iint_{S^{d-1} \times S^{d-1}} (\sqrt{f'} - \sqrt f)^2 b(\sigma' \cdot \sigma) \dd \sigma ' \dd \sigma \\
&\leq \frac 12 \iint_{S^{d-1} \times S^{d-1}} (f'-f) \log(f'/f) b(\sigma' \cdot \sigma) \dd \sigma ' \dd \sigma \\
&= \IB(f).
\end{align*}

In the cases where we can establish the inequality \eqref{e:spherical-inequality} with some constant $\Lambda_b > 0$ in this paper, we will actually be proving the following slightly stronger inequality.
\[ \int_{S^{d-1}} \Gamma^2_{\B,\Delta} (\log f,\log f) f \dd \sigma \geq 2\Lambda_b \iint_{S^{d-1} \times S^{d-1}} (\sqrt{f'} - \sqrt f)^2 b(\sigma' \cdot \sigma) \dd \sigma ' \dd \sigma.
 \]
It remains to be seen if a similar inequality holds with $\IB$ on the right-hand side. In this paper, we do not prove an inequality of the form
\begin{equation} \label{e:fiction} 
\int_{S^{d-1}} \Gamma^2_{\B,\Delta} (\log f,\log f) f \dd \sigma \geq \tilde \Lambda_b \IB(f),
\end{equation}
for any constant $\tilde \Lambda_b > 0$.

\section{The inequality on the sphere implies the monotonicity of the Fisher information}
\label{s:generic}

In this section, we prove Theorem~\ref{t:main}. In order to do so, we recall that we split the Fisher information in three terms $I_{parallel}$, $I_{radial}$ and $I_{spherical}$, see \eqref{e:fisher-split}. We already obtained suitable expressions for the derivatives of the first two in the direction of $Q(F)$. We are left with differentiating $I_{spherical}$.

\begin{lemma}[Differentiating $I_{spherical}$] \label{l:spherical-diffusion}
The following identity holds,
\[ 
\langle I'_{spherical}(F), Q(F) \rangle = -2 \iint \alpha(r) r^{d-3} \int F(z,r,\sigma) \Gamma^2_{\B,\Delta}(\log F(z,r,\cdot),\log F(z,r,\cdot)) \dd \sigma \dd r \dd z.
 \]
\end{lemma}

\begin{proof}
For each fixed value of $z \in \R^d$ and $r > 0$, we consider the function $F(z,r,\cdot): S^{d-1} \to (0,\infty)$. The gradient $\nabla_\sigma F$ corresponds to the gradient of this function on the tangent space of the sphere. The value of the operator $Q(F)$ defined in \eqref{e:Q} is the same as $\alpha(r) \B[F(z,r,\cdot)](\sigma)$.

With this interpretation, Lemma \ref{l:spherical-diffusion} is the same as Lemma \ref{l:fisher-along-B}.
\end{proof}

\begin{proof}[Proof of Theorem \ref{t:main}]
  In view of \eqref{e:fisher-split} and Lemmas~\ref{l:Iparallel}, \ref{l:Iradial} and \ref{l:spherical-diffusion}, we have
  \[2\langle I'(F),Q(F) \rangle \le -D_{parallel} -D_{spherical} + R\]
  with
\begin{align*}
D_{parallel} &=  \iiiint \alpha \left| \nabla_z \log F' - \nabla_z \log F \right|^2 (F'+F) \ b(\sigma'\cdot\sigma) r^{d-1} \dd \sigma' \dd \sigma \dd r \dd z, \\
D_{spherical} &= 4 \iint \alpha(r) r^{d-3} \int F(z,r,\sigma) \left[\Gamma^2_{\B,\Delta}(\log F(z,r,\cdot),\log F(z,r,\cdot))\right](\sigma) \dd \sigma \dd r \dd z, \\
R &= \iiiint \frac{(\dalpha)^2}{\alpha} \frac{(F - F')^2} {F'+F} \ b(\sigma'\cdot\sigma) r^{d-1} \dd \sigma' \dd \sigma \dd r \dd z.
\end{align*}

We now claim that $D_{spherical} \geq R$. To verify this, for each value of $z$ and $r$, we apply the inequality \eqref{e:spherical-inequality} to $f(\sigma) = F(z,r,\sigma)$. Using the assumption $r^2 \dalpha(r)^2/\alpha^2 \le 4 \Lambda_b$, we obtain $D_{spherical} \geq R$. This proves \eqref{e:aim} and consequently also Theorem \ref{t:main}.
\end{proof}

In order for Theorem \ref{t:main} to be useful, we need explicit lower bounds on $\Lambda_b$ that can be computed in practice. We will justify \eqref{e:spherical-inequality} by combining the following three inequalities.
\begin{align}
\int_{S^{d-1}} \Gamma^2_{\B,\Delta} (\log f,\log f) f \dd \sigma &\geq \cOne \int_{S^{d-1}} \frac{|\nabla f|^2}f \dd \sigma. \label{e:DtoFisher} \\
\cP \int_{S^{d-1}} |\nabla f|^2 \dd \sigma &\geq \iint_{S^{2} \times S^{2}} (f(\sigma') - f(\sigma))^2 b(\sigma' \cdot \sigma) \dd \sigma' \dd \sigma. \label{e:H1toNonlocal} \\
\iint_{S^{d-1} \times S^{d-1}} (\sqrt{f(\sigma')} - \sqrt{f(\sigma)})^2 b(\sigma' \cdot \sigma) \dd \sigma' \dd \sigma &\geq \frac 12 \iint_{S^{d-1} \times S^{d-1}} \frac{(f(\sigma') - f(\sigma))^2} {f(\sigma')+f(\sigma)} b(\sigma' \cdot \sigma) \dd \sigma ' \dd \sigma. \label{e:squareroots}
\end{align}

\begin{lemma} \label{l:squareroots}
For any function $f : S^{d-1} \to (0,\infty)$, the inequality \eqref{e:squareroots} holds.
\end{lemma}

\begin{proof}
The integrands satisfy the inequality pointwise. This is a consequence of the elementary inequality
\( (\sqrt a - \sqrt b)^2 \geq \frac 12 \frac{(a-b)^2}{a+b},\)
which holds for any values of $a,b > 0$.
\end{proof}

\begin{lemma} \label{l:three-inequalities}
Assume that \eqref{e:DtoFisher} and \eqref{e:H1toNonlocal} hold for any positive function $f:S^{d-1} \to (0,\infty)$ such that $f(\sigma) = f(-\sigma)$. Then \eqref{e:spherical-inequality} also holds with $\Lambda_b \geq 2 \cOne / \cP$.
\end{lemma}

\begin{proof}
We apply \eqref{e:DtoFisher}, then \eqref{e:H1toNonlocal} to $\sqrt{f}$, and finally \eqref{e:squareroots}.
\begin{align*}
\int_{S^2} \Gamma^2_{\Delta,\B} (\log f,\log f) f \dd \sigma &\geq \cOne \int_{S^{d-1}} \frac{|\nabla f|^2}f \dd \sigma \\
&= 4\cOne \int_{S^{d-1}} \left|\nabla \sqrt f \right|^2 \dd \sigma \\
&\geq \frac{4\cOne}{\cP} \iint_{S^{2} \times S^{2}} \left(\sqrt{f(\sigma')} - \sqrt{f(\sigma)}\right)^2 b(\sigma' \cdot \sigma) \dd \sigma  \\
&\geq \frac{2\cOne}{\cP} \iint_{S^{2} \times S^{2}} \frac{(f(\sigma') - f(\sigma))^2} {f(\sigma')+f(\sigma)} b(\sigma' \cdot \sigma) \dd \sigma ' \dd \sigma. \qedhere
\end{align*}
\end{proof}

The purpose of the next sections is to provide a proof of \eqref{e:spherical-inequality} by proving \eqref{e:DtoFisher} and \eqref{e:H1toNonlocal} for various assumptions on the kernel $b$. As we will see, we can obtain the inequality \eqref{e:H1toNonlocal} in great generality with relatively precise constants $\cP$ (see Propositions \ref{p:poincare} and \ref{p:poincare-subordinate} below). The inequality \eqref{e:DtoFisher} is harder to prove. We provide an example in Section \ref{s:counterexample} of a kernel $b$ on $S^1$ where it fails for any $\cOne > 0$. We prove in Section \ref{s:curvature} that it holds in dimension $d \geq 3$ for some constant $\cOne>0$. Under stronger structural assumptions on the kernel $b$, we prove in Section \ref{s:subordinate} that both \eqref{e:DtoFisher} and \eqref{e:H1toNonlocal} hold in any dimension with relatively precise constants $\cOne$ and $\cP$.

The inequality \eqref{e:DtoFisher} can be used to derive an integro-differential version of the log-Sobolev inequality on the sphere in the following way.

\begin{prop}[A $\Gamma^2$-criterion for an integral log-Sobolev inequality] \label{p:integral-log-sobolev}
  The inequality \eqref{e:DtoFisher} implies the following entropy-production/entropy inequality \[ \IB(f) \geq 2 \cOne \left( \int_{S^{d-1}} f \log f \dd \sigma - \left( \int_{S^{d-1}} f \dd \sigma \right) \log \left( \fint_{S^{d-1}} f \dd \sigma \right) \right).\]
  Here, $\IB$ is the entropy-production by $\B$ explained in Definition \ref{d:bfisher}.
\end{prop}
\begin{proof}
We proceed as in the proof by Bakry and \'Emery of the log-Sobolev inequality. We solve the heat equation with initial data $f$.
\begin{align*}
u(0,\sigma) &= f(\sigma), \\
\partial_t u &= \Delta u.
\end{align*}
Note that $u(t,\sigma)$ converges to the average of $f$ as $t \to \infty$.

We proceed using that $\partial_t \hs(u) = \is(u)$ and Lemma \ref{l:BFisher-along-Delta}, together with the inequality \eqref{e:DtoFisher},
\begin{align*}
\partial_t \hs(t) &= \is(t) \\
&\leq \frac 1 \cOne \int_{S^2} \Gamma^2_{\Delta,\B} (\log f,\log f) f \dd \sigma \\
&= -\frac 1 {2\cOne} \partial_t \IB.
\end{align*}
Integrating the inequality on $t \in (0,\infty)$, we conclude the proof.
\end{proof}

In Corollary \ref{c:curvature-lower-bound} and Lemma \ref{l:integratingBEintime}, we compute the value of $\cOne$ for some classes of kernels $b$.

\section{Lower bounds for any collision kernel using the curvature of the sphere}
\label{s:curvature}

In the previous section, we proved the monotoncity of the Fisher information of any solution $f$ of \eqref{e:boltzmann} under the assumption of Theorem \ref{t:main}. In order for this theorem to be applicable to any collision kernel of physical relevance, we have to compute explicit lower bounds for the constant $\Lambda_b$ in \eqref{e:spherical-inequality}. Studying this constant $\Lambda_b$ is the purpose of the rest of the paper. In this section we develop some tools to take advantage of the curvature of the sphere and eventually prove that $\Lambda_b \geq d-2$.

Given any $x \perp \sigma$ and $y \perp \sigma'$. We observe that $|y-x|_{\sigma',\sigma}^2$ can be written as the sum of two positive terms
\[ |y-x|_{\sigma',\sigma}^2 = |y - P_{\sigma,\sigma'} x|^2 + \left( |x|^2 - |P_{\sigma,\sigma'} x|^2 \right).\]
Recall that $|\cdot|_{\sigma',\sigma}$, $P_{\sigma,\sigma'}$ and $M_{\sigma,\sigma'}$ are defined in \eqref{e:nonlocalconnection}, \eqref{e:def-P}  and \eqref{e:def-M}, respectively. The second term in the previous equality is positive because of Lemma \ref{l:Pcontracts}.

In this section, we take advantage of this decomposition of $|y-x|_{\sigma',\sigma}^2$ to derive a generic lower bound for the expression in Lemma \ref{l:Gamma2-with-T} for $\Gamma^2_{\B,\Delta}$.

\begin{lemma}[Lower bound for the iterated carré du champ using curvature] \label{l:pre-curvature-bound}
Assume $d>2$. The following inequality holds for any function $f : S^{d-1} \to \R$.
\[ \Gamma^2_{\B,\Delta}(f,f) \geq \cOne |\nabla f|^2.\]
Here
\[ \cOne = \frac{d-2}{2(d-1)} \int_{S^{d-1}} (1 - (e \cdot \sigma' )^2) b(e \cdot \sigma') \dd \sigma',\]
is a constant independent of the choice of $e \in S^{d-1}$.
\end{lemma}

\begin{proof}
From the decomposition explained above, we know that
\begin{align*}
|y-x|_{\sigma',\sigma}^2 &\geq \left( |x|^2 - |P_{\sigma,\sigma'} x|^2 \right) \\
&= \bigg(1-(\sigma'\cdot \sigma)^2 \bigg) |\Pi_{\perp} x|^2,
\end{align*}
where $\Pi_\perp$ is the orthogonal projection to the codimension-2 space orthogonal to $\sigma'$ and $\sigma$ (recall Lemma~\ref{l:Pcontracts}).

We start from the expression in Lemma \ref{l:Gamma2-with-T} and apply this inequality to the integrand.
\begin{align*}
\Gamma_{\B,\Delta}^2(f,f) &= \frac 12 \int_{S^{d-1}} |\nabla f(\sigma') - \nabla f(\sigma)|_{\sigma',\sigma}^2 b (\sigma' \cdot \sigma) \dd \sigma' \\
&\geq \frac 12 \int_{S^{d-1}} \bigg(1 - (\sigma' \cdot \sigma)^2\bigg) \left|\Pi_\perp \nabla f(\sigma) \right|^2   b (\sigma' \cdot \sigma) \dd \sigma'.
\end{align*}

We express this last integal in spherical coordinates. We write $\sigma' = c \sigma + \omega$, with $c \in (-1,1) = \sigma' \cdot \sigma$ and $\omega \perp \sigma$. The value of $\omega$ belongs to the $(d-2)$-sphere of radius $(1-c^2)^{1/2}$ on the orthogonal complement of $\sigma$, which coincides with $T_\sigma S^{d-1}$. We observe that $\Pi_\perp \nabla f(\sigma) = \nabla f(\sigma) - (\omega \cdot \nabla f(\sigma)) \omega / |\omega|^2$. Thus,
\begin{align*}
\Gamma_{\B,\Delta}^2(f,f) &\ge \frac 12 \int_{-1}^1 \int_{\substack{\{ \omega \perp \sigma, \\ |\omega|^2 = 1-c^2\}}} (1-c^2) b(c) \left( |\nabla f(\sigma)|^2 - (\nabla f(\sigma) \cdot \omega)^2/(1-c^2) \right) \dd \omega \dd c \\
&= \cOne |\nabla f(\sigma)|^2. 
\end{align*}
Here,
\begin{align*}
\cOne &= \frac 12 \int_{-1}^1 \int_{\substack{\{ \omega \perp \sigma, \\ |\omega|^2 = 1-c^2\}}} (1-c^2) b(c) \left( 1 - (e \cdot \omega)^2/(1-c^2) \right) \dd \omega \dd c \\
&= \frac12 \left(\int_{S^{d-2}} \left( 1- (e  \cdot \bar \omega)^2 \right) \dd \bar \omega \right)  \int_{-1}^1 (1-c^2)^{\frac{d}2} b(c) \dd c \\
&= \frac{d-2}{d-1} \cdot \frac{\omega_{d-2}}2 \int_{-1}^1 (1-c^2)^{d/2} b(c) \dd c.\qedhere
\end{align*}
\end{proof}

\begin{cor}[Fisher bounds from below the integrated iterated carré du champ operator]  \label{c:curvature-lower-bound}
Assume $d>2$. For any function $f : S^{d-1} \to (0,\infty)$,  inequality~\eqref{e:DtoFisher} holds,
\[ \int_{S^{d-1}} f \Gamma^2_{\B,\Delta}(\log f,\log f) \dd \sigma \geq \cOne \int_{S^{d-1}} \frac{|\nabla f|^2} f \dd \sigma.\]
Here, for any $e \in S^{d-1}$,
\[ \cOne = \frac{d-2}{2(d-1)} \int_{S^{d-1}} (1 - (\sigma' \cdot e)^2) b(\sigma'\cdot e) \dd \sigma'.\]
\end{cor}

\begin{proof}
We integrate the inequality from Lemma \ref{l:pre-curvature-bound}.
\end{proof}

The next step is to control the right hand side of \eqref{e:spherical-inequality} using the usual Fisher information. We do it precisely in the next section. 

\section{A Poincaré-type inequality}
\label{s:poincare}

We will prove the following estimate that provides an upper bound for integral Dirichlet forms in terms of the $\dot{H}^1$ norm of a function on the sphere. We will later apply this lemma to $\sqrt f$ in order to get estimates involving the Fisher information.

\begin{prop}[A Poincaré-type inequality] \label{p:poincare}
For any function $f : S^{d-1} \to (0,\infty)$ such that $f(\sigma) = f(-\sigma)$, the following inequality holds
\begin{equation}\label{e:poincare}
  \cP \int_{S^{d-1}} |\nabla f|^2 \dd \sigma \geq \iint_{S^{d-1} \times S^{d-1}} (f'-f)^2 b(\sigma'\cdot\sigma) \dd \sigma' \dd \sigma.
\end{equation}
Here, the constant $\cP$ is given for any fixed $e \in S^{d-1}$ by the following formula,
\[ \cP = \frac 1 {d-1} \int_{S^{d-1}} (1 - (\sigma \cdot e)^2)   b (e \cdot \sigma) \dd \sigma. \]
\end{prop}

\begin{remark}
When the kernel $b$ is constant, the inequality \eqref{e:poincare} is exactly the Poincaré inequality on the sphere for functions with average zero.
\end{remark}

\subsection{Spectral decomposition}

The operators $-\Delta$ and $\B$ are self-adjoint and commute. Thus, they have a common base of eigenfunctions. Recall that the eigenfunctions of $-\Delta$ are spherical harmonics, corresponding to the restriction to $S^{d-1}$ of homogeneous harmonic polynomials in $\R^d$. Spherical harmonics have been studied extensively. A modern reference for the subject is \cite{efthimiou2014}.

Generically, we expect $\Delta$ and $\B$ to have the same eigenspaces. We will prove only one inclusion. That is, we prove that all eigenfunctions of $-\Delta$ with eigenvalue $\lambda$ are eigenfunctions of $\B$ with the same eigenvalue $\tilde \lambda$.

\begin{lemma}[Eigenspaces of $\Delta$ and $\B$] \label{l:eigenspaces}
Each eigenspace of $\Delta$ (which is a finite dimensional space of spherical harmonics) is contained in some eigenspace of $\B$.
\end{lemma}

\begin{proof}
It is well known that the eigenspace corresponding the eigenvalue $\lambda_\ell = \ell (\ell+d-2)$ for $-\Delta$ consists in the restriction to the sphere of $\ell$-homogeneous harmonic polynomials. It is a finite dimensional space of smooth functions.

Since $\B$ and $\Delta$ commute, this finite-dimensional eigenspace is an invariant space by $\B$.

Observe also that both $\Delta$ and $\B$ commute with rotations. Thus, axially symmetric functions (which are those that depend only on $e_1 \cdot \sigma$ on $S^{d-1}$) are invariant by both $\Delta$ and $\B$. We may restrict these two operators to the space of axially symmetric functions and repeat the above analysis. If an axis is chosen arbitrarily, the spherical harmonics of order $\ell$ that are axially symmetric around that axis is a subspace of dimension one (see \cite[Theorem 4.10]{efthimiou2014}). They consist of functions of the form
\[ Y_\ell(\sigma) = a L_\ell(\sigma \cdot e_1),\]
for $a \in \R$ and $L_\ell$ being the Legendre polynomial of degree $\ell$. Since this eigenspace of $\Delta$ has dimension one, and $\B$ commutes with $\Delta$, then it is also a one-dimensional eigenspace for $\B$. Thus, the function $Y_\ell$ above is an eignfunction for $\B$.

It is well known that any generic (not necessarily axially symmetric) spherical harmonic $Y$ of order $\ell$ can be written as a linear combination of rotations of the function $Y_\ell$ above (see \cite[Theorem 4.13]{efthimiou2014}). Since any rotation of $Y_\ell$ is going to be also an eigenfunction for $\Delta$ and $\B$ with the same eigenvalue, then the full eigenspace of spherical harmonics of order $\ell$ are eigenfunctions of $\B$ with the same eigenvalue.
\end{proof}

Let us recall that the eigenvalue $\lambda_\ell$ of $-\Delta$ corresponding to spherical harmonics of order $\ell$ equals $\ell (\ell + d -2)$. Let us write $\tilde \lambda_\ell$ for the corresponding eigenvalue of $-\B$ for spherical harmonics of order $\ell$. Given any generic function $f : S^{d-1} \to \R$, we write its expansion in terms of spherical harmonics.
\[ f = \sum_{\ell=0}^{+\infty} x_\ell Y_\ell.\]
Here, each $Y_\ell$ is a spherical harmonic of order $\ell$. Observe that
\begin{align*}
\int_{S^{d-1}} |\nabla f|^2 \dd \sigma &= \int_{S^{d-1}} (-\Delta f) f \dd \sigma, \\
\iint_{S^{d-1} \times S^{d-1}} (f'-f)^2 b(\sigma'\cdot\sigma) \dd \sigma' \dd \sigma &= 2\int_{S^{d-1}} (-\B f) f \dd \sigma.
\end{align*}
Therefore
\begin{align*}
\int_{S^{d-1}} |\nabla f|^2 \dd \sigma &= \sum_{\ell=0}^{+\infty} \lambda_\ell x_\ell^2,  \\
\iint_{S^{d-1} \times S^{d-1}} (f'-f)^2 b(\sigma'\cdot\sigma) \dd \sigma' \dd \sigma &= \sum_{\ell=0}^{+\infty} 2 \tilde \lambda_\ell x_\ell^2.
\end{align*}

The first term on the right-hand side vanishes, so we may exclude it from our analysis.

Recall that $Y_\ell$ is the restriction to the sphere of a homogeneous harmonic polynomial of degree $\ell$. In particular, $Y_\ell$ is an odd function if $\ell$ is odd. When the function $f$ is even (i.e. it satisfies $f(-\sigma) = f (\sigma)$ for all $\sigma$), then $x_{2\ell+1}=0$ for all $\ell \in \mathbb{N}$. 

The following lemma follows from the discussion above.
\begin{lemma}[First formula for the constant $\cP$] \label{l:C2}
The inequality~\eqref{e:poincare} holds true with
  \begin{equation}
  \label{e:C2}
  \cP = 2 \sup_{\ell=1,2,3,\dots} \frac{\tilde \lambda_{2\ell}}{\lambda_{2\ell}}.
\end{equation}
\end{lemma}

\subsection{An inequality for Legendre polynomials}

In order to get a useful formula for the constant $\cP$ appearing in \eqref{e:poincare}, we are going to use
a pointwise estimate on Legendre polynomials of even order. 

We start with a formula for the eigenvalues $\tilde \lambda_\ell$ of $-\B$.
\begin{lemma} \label{l:tildelambda}
The eigenvalue $\tilde \lambda_\ell$ for $-\B$ corresponding to spherical harmonics of order $\ell$ is given by the formula.
\[ \tilde \lambda_\ell = \int_{S^{d-1}} \left(1- P_\ell(e_1\cdot\sigma') \right) b(e_1\cdot\sigma') \dd \sigma', \]
where $P_\ell$ is the Legendre polynomial of order $\ell$ with the normalization $P_\ell(1) = 1$.
\end{lemma}

\begin{proof}
We recall that the unique axially symmetric spherical harmonic $Y_\ell$ of order $\ell$ such that $Y_\ell(e_1) = 1$ is given by
$$Y_\ell (\sigma) = P_\ell (e_1 \cdot \sigma).$$
Therefore, for this function $Y_\ell$, we have $-\Delta Y_\ell(e_1) = \lambda_\ell$ and
\begin{align*}
\tilde \lambda_\ell &= -\B Y_\ell(e_1) \\
&= \int_{S^{d-1}} \left(1- P_\ell(e_1\cdot\sigma') \right) b(e_1\cdot\sigma') \dd \sigma'. \qedhere
\end{align*}
\end{proof}

\begin{prop}[An inequality for Legendre polynomials] \label{p:legendre}
Given the Legendre polynomials $P_{\ell}$ normalized so that $P_\ell(1)=1$, we have for all $x \in [-1,1]$ and $\ell \geq 1$,
\begin{equation} \label{e:aim-Legendre}
1-P_{2\ell}(x) \leq \frac{\lambda_{2\ell}}{\lambda_2} (1-P_2(x)).
\end{equation}
Here $\lambda_\ell = \ell(d+\ell-2)$ is the eigenvalue of the spherical Laplacian corresponding to $Y_\ell(\sigma) := P_\ell(e_1 \cdot \sigma)$.
\end{prop}

Legendre polynomials $P_\ell(x)$ correspond to homogeneous harmonic functions in $\R^d$ which are invariant by rotations over the $e_1$ axis. In 2D, these homogeneous harmonic functions can easily be computed using complex analysis as the real part of $(x+iy)^\ell$. In dimension $d>2$, we may start with one of these 2D profiles written in terms of $x_1$ and $x_2$, and average over all possible rotations over the $e_1$ axis. We are left with the following formula that appears in \cite[Theorem 4.26]{efthimiou2014}.

\begin{lemma}[{Integral representation of Legendre polynomials --\cite[Theorem 4.26]{efthimiou2014}}] \label{l:2Daverages}
Assuming $d \geq 3$, the following equality holds
\[ P_\ell(x) = \frac{\omega_{d-3}}{\omega_{d-2}} \int_{-1}^1 \left( x + is\sqrt{1-x^2} \right)^\ell (1-s^2)^{\frac{d-4}2} \dd s.\]
\end{lemma}

It is useful to understand the weight in the integral. For any $p > 1$, the following identity holds,
$$ \omega_{p-1} \int_{-1}^1 (1-s^2)^{\frac{p-2}2} \dd s = \omega_{p}. $$
Here, $\omega_p$ is the surface area of $S^{p}$. In fact, all the quantities involved are computable.
We recall that (see for instance \cite[Proposition~2.3]{efthimiou2014}),
$$ \omega_p = \frac{2 \pi^{d/2}}{\Gamma(d/2)}. $$
In particular,
$$ \int_{-1}^1 (1-s^2)^{\frac{p-2}2} \dd s = \sqrt \pi \frac{\Gamma(p/2)}{\Gamma((p+1)/2)}.$$

It is easy to verify that the right hand side of Lemma \ref{l:2Daverages} is a real number by averaging the values of the integral at $s$ and $-s$.

Observe that we have $P_\ell(1) = 1$. In view of this, we write
\begin{equation} \label{e:1-P}
1-P_\ell(x) = \frac{\omega_{d-3}}{\omega_{d-2}} \int_{-1}^1 \left(1-\left( x + is\sqrt{1-x^2} \right)^\ell\right) (1-s^2)^{\frac{d-4}2} \dd s.
\end{equation}

\begin{proof}[Proof of Proposition~\ref{p:legendre}]
  The proof proceeds in several steps.

\paragraph{\textsc{Step 1.}}
  We first compute explicitely the right hand side of \eqref{e:aim-Legendre}. 
The polynomial $P_2$ can be computed explicitly,
\[ P_2(x) = \frac{d}{d-1} x^2 - \frac 1 {d-1}.\]
We quickly justify this classical fact. The function $Y_2$ is the restriction to the sphere of the unique homogeneous quadratic harmonic polynomial that is invariant by any rotation fixing $e_1$ and $Y_2(e_1) = 1$. Thus, it has to be $Y_2(x) = x_1^2 - (x_2^2 + \dots +x_d^2)/(d-1)$, and then $P_2(x) = x^2 + (1-x^2) / (d-1)$. In particular,
\begin{equation}\label{e:P2}
  1-P_2(x) = \frac{d}{d-1} (1-x^2).
\end{equation}

Since $\lambda_2 = 2d$, the inequality we want to prove reduces to
\begin{equation} \label{e:aim2}
1-P_{2\ell}(x) \leq \frac{\lambda_{2\ell}}{2(d-1)} (1-x^2) = \frac{\ell (d-2+2\ell)} {(d-1)} (1-x^2).
\end{equation}
Our plan is to use \eqref{e:1-P} to prove \eqref{e:aim2}. To get there, we need to prove a series of lemmas first.
\bigskip

\paragraph{\textsc{Step 2. The 2D case.}}

When $d=2$, the polynomials $P_\ell(x)$ correspond to $\cos(\ell \theta)$ if $x = \cos \theta$. The inequality \eqref{e:aim-Legendre} that we want to prove takes the form
\[ \ell^2 \geq \frac{1-\cos(2\ell \theta)}{1-\cos(2\theta)}. \]

This inequality can be proved relatively easily using trigonometric identities and induction. We observe that $1-\cos(2\theta) = 2 \sin(\theta)^2$. Therefore
\[ \frac{1-\cos(2\ell \theta)}{1-\cos(2\theta)} = \frac{\sin(\ell \theta)^2}{\sin(\theta)^2}. \]

The following inequality is easy to prove using induction on $\ell$ and the formulas for the sine of a sum of the angles.
\[ \frac{|\sin(\ell \theta)|}{|\sin(\theta)|} \leq \ell. \]
Indeed, the inequality is obvious for $\ell=1$. Once we know it for any value of $\ell$, we use the formula for the sine of sums to get
\begin{align*} 
|\sin((\ell+1) \theta)| &= |\sin((\ell) \theta) \cos(\theta) + \cos((\ell) \theta) \sin(\theta)| \\
&\leq |\sin((\ell) \theta)| + |\sin(\theta)| \\
&\leq (\ell+1) |\sin(\theta)|.
\end{align*}

These two steps that are used to prove the inequality in 2D correspond to the first two lemmas that are presented in the general case. Without the 2D case, their statements would look admittedly mysterious.
\bigskip

\paragraph{\textsc{Step 3. Dimension greater than 2.}}

\begin{lemma} \label{l:2}
For any $\ell = 1,2,3,\dots$, the following identity holds for $x,s \in [-1,1]$,
\[ \Re \left(1-\left( x + is\sqrt{1-x^2} \right)^{2\ell}\right) = 2 \left\{ \Im \left( x + is \sqrt{1-x^2} \right)^\ell \right\}^2 + 1-\left( x^2 + s^2 (1-x^2) \right)^\ell. \]
\end{lemma}

\begin{proof}
For any complex number $z = a+ib$, we observe that
\begin{align*}
1 - \Re z^2 = 1 - a^2 + b^2 = 1-|z|^2 + 2b^2.
\end{align*}

If we set $z = \left( x + is\sqrt{1-x^2} \right)^{\ell}$, we conclude the statement of the lemma.
\end{proof}

\begin{lemma} \label{l:1}
The following inequality holds for all $\ell = 1,2,3,\dots$ and $x,s \in [-1,1]$,
\[ \left\vert \Im \left( x + is \sqrt{1-x^2} \right)^\ell \right\vert \leq \ell |s| \sqrt{1-x^2}.\]
\end{lemma}

\begin{proof}
We prove it by induction. It is a trivial identity for $\ell = 1$. Then, we observe
\begin{align*}
\left\vert \Im \left( x + is \sqrt{1-x^2} \right)^{\ell+1} \right\vert &= \left\vert \Im \left( x + is \sqrt{1-x^2} \right)^\ell x + \Re \left( x + is \sqrt{1-x^2} \right)^\ell s \sqrt{1-x^2} \right\vert \\
&\leq \left\vert \Im \left( x + is \sqrt{1-x^2} \right)^\ell \right\vert |x| + \left\vert \Re \left( x + is \sqrt{1-x^2} \right)^\ell \right\vert |s| \sqrt{1-x^2} \\
&\leq \left\vert \Im \left( x + is \sqrt{1-x^2} \right)^\ell \right\vert + |s| \sqrt{1-x^2} \\
\intertext{and then we use the inductive hypothesis,}
&\leq (\ell+1)|s|\sqrt{1-x^2}. \qedhere
\end{align*}
\end{proof}

Combining Lemmas \ref{l:1} and \ref{l:2}, we obtain
\[ \Re \left(1-\left( x + is\sqrt{1-x^2} \right)^{2\ell}\right) \leq 2\ell^2 s^2 (1-x^2) + 1-\left( x^2 + s^2 (1-x^2) \right)^\ell. \]
Using the convexity of the function $(1-z)^\ell$ for $z \in (0,1)$, we have
  \[ 1 - (1-z)^\ell \leq \ell z .\]
  In particular, if we pick $z = (1-s^2) (1-x^2)$, we have
  \[ 1-\left( x^2 + s^2 (1-x^2) \right)^\ell \le \ell (1-s^2)(1-x^2).\]
  We thus get
\begin{align*} 
\Re \left(1-\left( x + is\sqrt{1-x^2} \right)^{2\ell}\right) &\leq 2\ell^2 s^2 (1-x^2) + (1-x^2) \ell (1-s^2) \\
&= (1-x^2) \ell (1-s^2 +2 \ell s^2).
\end{align*}

Finally, we integrate in $s$ to prove \eqref{e:aim2} using \eqref{e:1-P}.
\begin{align*}
1-P_{2\ell}(x) &= \Re \frac{\omega_{d-3}}{\omega_{d-2}} \int_{-1}^1 \left(1-\left( x + is\sqrt{1-x^2} \right)^{2\ell} \right) (1-s^2)^{\frac{d-4}2} \dd s \\
&\leq \frac{\omega_{d-3}}{\omega_{d-2}} \int_{-1}^1 (1-x^2) \ell (1-s^2+2\ell s^2) (1-s^2)^{\frac{d-4}2} \dd s \\
&= \frac{\omega_{d-3}}{\omega_{d-2}} \ell (1-x^2) \int_{-1}^1 \left\{ (1-2\ell) (1-s^2)^{\frac{d-2}2} + 2 \ell (1-s^2)^{\frac{d-4}2} \right\} \dd s \\
&= \frac{\omega_{d-3}}{\omega_{d-2}} \ell (1-x^2) \left( (1-2\ell) \sqrt \pi \frac{\Gamma(d/2)}{\Gamma((d+1)/2)} + 2 \ell \sqrt \pi \frac{\Gamma(d/2-1)}{\Gamma((d+1)/2-1)} \right) \\
&= \ell \left( (1-2\ell) \frac{d-2}{d-1} + 2 \ell \right) \ (1-x^2) \\
&= \frac{\ell (d-2+2\ell)} {(d-1)} (1-x^2).
\end{align*}

And that is the inequality we wanted to prove.
\end{proof}

\begin{cor}[Second formula for $\cP$]\label{c:C2}
  For $\ell=1,2,3,\dots$, we have $\frac{\tilde \lambda_{2\ell}}{\lambda_{2\ell}} \le \frac{\tilde \lambda_2}{\lambda_2}$. In other words, $\cP$ defined by \eqref{e:C2} equals  $2\frac{\tilde \lambda_2}{\lambda_2}$.  
\end{cor}
\begin{proof}
In view of Lemma~\ref{l:tildelambda}, we can use Proposition~\ref{p:legendre} to estimate $\tilde{\lambda}_{2\ell}$ from above, 
\begin{align*}
  \tilde{\lambda}_{2\ell} & = \int_{S^{d-1}} (1 - P_{2\ell} (e_1 \cdot \sigma')) b(e_1 \cdot \sigma') \dd \sigma' \\
  & \le \frac{\lambda_{2\ell}}{\lambda_2} \int_{S^{d-1}} (1 - P_{2} (e_1 \cdot \sigma')) b(e_1 \cdot \sigma') \dd \sigma' \\
  & = \frac{\lambda_{2\ell}}{\lambda_2} \tilde \lambda_2
\end{align*}
where we used Lemma~\ref{l:C2} to get the last line. 
\end{proof}

We are now ready to prove Proposition~\ref{p:poincare}.
\begin{proof}[Proof of Proposition~\ref{p:poincare}]
We combine Lemmas~\ref{l:C2} with Corollary~\ref{c:C2} and get that \eqref{e:poincare} holds with $\cP = 2\frac{\tilde \lambda_2}{\lambda_2}$.

It is well known that $\lambda_2 = 2d$. To compute $\tilde \lambda_2$, we use Lemma \ref{l:tildelambda} and \eqref{e:P2},
\begin{align*}
\tilde \lambda_2 &= \int_{S^{d-1}} \left(1- P_2(e_1\cdot\sigma') \right) b(e_1\cdot\sigma') \dd \sigma', \\
&=  \int_{S^{d-1}} \frac d {d-1} \left( 1- (e_1\cdot\sigma')^2 \right) b(e_1\cdot\sigma') \dd \sigma'.
\end{align*}
And the formula for $\cP = 2\frac{\tilde \lambda_2}{\lambda_2}$ follows.
\end{proof}

\subsection{The key inequality on the sphere for large dimensions}

\begin{prop}[The key inequality on the sphere in large dimensions] \label{p:spherical-inequality-from-curvature}
Assume $d>2$. For any kernel $b$, the inequality \eqref{e:spherical-inequality} holds with \( \Lambda_b \geq  d-2. \)
\end{prop}

\begin{proof}
Corollary \ref{c:curvature-lower-bound} establishes the inequality \eqref{e:DtoFisher} for 
\[ \cOne = \frac{d-2}{2(d-1)} \int_{S^{d-1}} (1 - (\sigma' \cdot e)^2) b(\sigma'\cdot e) \dd \sigma'.\]
Proposition~\ref{p:poincare} establishes the inequality \eqref{e:H1toNonlocal} for 
\[ \cP = \frac 1 {d-1} \int_{S^{d-1}} (1 - (\sigma \cdot e)^2)   b (e \cdot \sigma) \dd \sigma. \]
Therefore, we prove \eqref{e:spherical-inequality} using Lemma \ref{l:three-inequalities} with 
\[ \Lambda_b = 2 \cOne/\cP = d-2. \qedhere\]
\end{proof}

\section{Example of a singular collision kernel}
\label{s:counterexample}

In this section, we show an example of a collision kernel $b$ for $d=2$ so that the inequality \eqref{e:spherical-inequality} does not hold for any $\Lambda_b > 0$. 

The example consists in a very singular kernel $b$ which is a Dirac delta concentrated at the deviation angle $\pi/2$. We will see that for this collision kernel the inequality \eqref{e:spherical-inequality} can only hold with $\Lambda_b = 0$.
\begin{lemma}[A counter-example in 2D]\label{l:counter-example}
For a given $\Lambda>0$ arbitrarily large, there exists a $\pi$-periodic function $f$ such that the following inequality is false, 
\begin{equation}\label{e:counter-example}
  \begin{aligned}
    \int_{[0,2\pi]} &|\partial_\theta \log f(\theta+\pi/2) - \partial_\theta \log f(\theta)|^2 (f(\theta+\pi/2)+f(\theta)) \dd \theta  \\ \qquad &\geq \Lambda \int_{[0,2\pi]} \frac{(f(\theta+\pi/2) - f(\theta))^2} {f(\theta+\pi/2)+f(\theta)}  \dd \theta .
 \end{aligned}
\end{equation}
\end{lemma}
\begin{proof}
In order to do so, we consider $f = (1 + A \psi) h$ with $\psi,h$ that are smooth, $\pi$-periodic. We also choose $h$ so that,
\[ h(\theta) = \begin{cases} 2 & \text{ if } \theta \in [\pi/4,\pi/2] \\ 1 & \text{ if } \theta \in [3 \pi/4,\pi] \end{cases} \]
and $h$ increasing in $[0,\pi]$ and decreasing in $[\pi,2\pi]$. As far as $\psi$ is concerned, we choose it non-negative, vanishing on $[0,\pi/4]$ and $\pi/2$-periodic.

Since $\psi$ is $\pi/2$-periodic and $h$ has zero derivative on the support of $\psi$, we observe that
\begin{align*}
  \log f(\theta+\pi/2) - \log f(\theta) &= \log h(\theta+\pi/2) - \log h(\theta), \\
 \partial_\theta \log h(\theta+\pi/2) - \partial_\theta \log h(\theta) &= 0 \text{ on the support of } \psi.
\end{align*}
These two facts imply that 
\[ \int_{[0,2\pi]} |\partial_\theta \log f(\theta+\pi/2) - \partial_\theta \log f(\theta)|^2 (f(\theta+\pi/2)+f(\theta)) \dd \theta \text{ does not depend on $A$}.\]

But 
\begin{align*} \int_{[0,2\pi]} \frac{(f(\theta+\pi/2) - f(\theta))^2} {f(\theta+\pi/2)+f(\theta)}  \dd \theta
  & =  \int_{[0,2\pi]} (1+A \psi (\theta)) \frac{(h(\theta+\pi/2) - h(\theta))^2} {h(\theta+\pi/2)+h(\theta)}  \dd \theta \\
  & \ge A \int_{[0,2\pi]}  \psi (\theta) \frac{(h(\theta+\pi/2) - h(\theta))^2} {h(\theta+\pi/2)+h(\theta)}  \dd \theta.
\end{align*}
Taking $A \to \infty$, we conclude that \eqref{e:counter-example} cannot hold true unless $\Lambda=0$. 
\end{proof}

\section{Subordinate diffusions}
\label{s:subordinate}

Our example of a two dimensional kernel for which $\Lambda_b = 0$ suggests that the lower bound $\Lambda \geq d-2$ may not be improvable in full generality. However, this lower bound also appears to be suboptimal in many cases of interest. In this section we study a special kind of kernels of the form \eqref{e:intro-subordinate}. It is a fairly generic class that includes for example the kernels of the fractional Laplacian on the sphere. We obtain explicit lower bounds for $\Lambda_b$ for kernels $b$ in this class that are always larger than the dimension $d$. These improved estimates for kernels of subordinate processes is what eventually allows us to verify the assumptions of Theorem \ref{t:main} for power-law kernels in the very soft potential range.

The objective of this section is to prove \eqref{e:spherical-inequality} for kernels $b$ that are of the form
\begin{equation} \label{e:subordination}
b(c) = \int_0^\infty h_t(c) \omega(t) \dd t.
\end{equation}
Here, $h_t$ is the heat kernel on $S^{d-1}$ after time $t>0$, and $\omega \geq 0$ is an arbitrary weight function. We will obtain that for these kernels, the inequality \eqref{e:spherical-inequality} holds with a constant $\Lambda_b$ that is always larger than the dimension $d$.

In this work, it  is convenient to call  heat kernel the function $h_t$  so that
\[ u(t,\sigma) = \int_{S^{d-1}} f(\sigma') h_t(\sigma\cdot\sigma') \dd \sigma'\]
is the representation formula for the heat equation on the sphere,
\begin{align*}
u(0,\sigma) &= f(\sigma), \\
\partial_t u &= \Delta u.
\end{align*}
By the rotational symmetry of the heat equation, the heat kernel $h_t$ depends only on the angle between $\sigma$ and $\sigma'$. We write it as $h_t(\sigma\cdot\sigma')$ to match the notation for the angular part of the collision kernel $b(\sigma\cdot\sigma')$. 

\subsection{Explicit formulas for constants in key inequalities}

We state first a proposition that improves Corollary~\ref{c:curvature-lower-bound} for kernels $b$ of the form \eqref{e:subordination}. In this case, we can prove that \eqref{e:DtoFisher} holds with a larger constant $\cOne$ and in any dimension.

\begin{prop}[Integrated $\Gamma^2_{\B,\Delta}$ criterion for subordinate diffusions] \label{p:cOne_for_subordinate}
Let $b$ be any kernel of the form \eqref{e:subordination}. For any function $f:S^{d-1} \to (0,\infty)$ such that $f(\sigma) = f(-\sigma)$, the inequality \eqref{e:DtoFisher} holds with
\[ \cOne = \frac 12 \int_0^\infty \omega(t) (1-\exp(-2\Lambda_{local} t)) \dd t.\]
\end{prop}

We then give a simpler proof of the Poincaré-type inequality from Proposition~\ref{p:poincare} for collision kernels of the form \eqref{e:subordination}. We also provide the reader with a handy formula for the constant $C_P$ in which the weight $\omega$ appears.
\begin{prop}[Poincaré-type inequality, again] \label{p:poincare-subordinate}
For any function $f : S^{d-1} \to (0,\infty)$ such that $f(\sigma) = f(-\sigma)$, the inequality~\eqref{e:H1toNonlocal} holds with
\[ \cP = \int_0^\infty \omega(t) \frac{1-e^{-2dt}}d \dd t. \]
\end{prop}

Finally, we prove \eqref{e:spherical-inequality} for kernels of the form \eqref{e:subordination}.
\begin{prop}[The key inequality on the sphere for subordinate diffusions] \label{p:main-for-subordinate}
Let $b$ be a kernel of the form \eqref{e:subordination} for some weight $\omega \geq 0$. The inequality \eqref{e:spherical-inequality} holds with 
\[ \Lambda_b \geq d \frac{\int_0^\infty \omega(t) (1-\exp(-2\Lambda_{local} t)) \dd t}{\int_0^\infty \omega(t) (1-e^{-2dt}) \dd t}. \]
In particular, $\Lambda_b > d$ for any kernel $b$ of the form \eqref{e:subordination}.
\end{prop}

\subsection{Comments on subordinate diffusions}

Recall that the fractional Laplacian, as well as a large family of integral operators, can be written in the form \eqref{e:subordination} for some weight function $\omega \geq 0$. In terms of the corresponding stochastic processes, we are analyzing the integral operators that arise as infinitesimal generators of a process obtained from the Brownian motion by subordination. For the kernel $b$ of the fractional Laplacian $-(-\Delta)^s$, we would have $\omega(\tau) = c_s \tau^{-1-s}$ for some constant $c_s$ that only depends on dimension and $s$.

The purpose of this section is to establish a proof of \eqref{e:spherical-inequality} for collision kernels $b$ of the form \eqref{e:subordination}. We do it by establishing \eqref{e:DtoFisher} and \eqref{e:H1toNonlocal} and applying Lemma \ref{l:three-inequalities}. Because of the example in Section \ref{s:counterexample}, we know that \eqref{e:spherical-inequality} (and in particular \eqref{e:DtoFisher}) cannot hold for just any kernel $b$. The condition \eqref{e:subordination} is the further structural assumption on the kernel that we make to be able to prove \eqref{e:DtoFisher} with a precise constants $\cOne$.

The kernels of the form \eqref{e:subordination} comprise a relatively general family of kernels. In $\R^d$, this type of kernels has been studied extensively. They correspond to the jump kernels of subordinate Brownian motions, and they are identified as kernels $k(x-y)$ that coincide with a completely monotone function applied to $|x-y|^2$. Indeed, the jump kernel is the density of the Lévy measure of the subordinate process. It is given by the following formula,
\[ k (z) = \mathcal{L} (|z|^2) = \int_0^{\infty} (4\pi s)^{-d/2} e^{-\frac{|z|^2}{4s}} \rho (\dd s) \]
(see  \cite[(30.8), p.198]{sato}). Now Bernstein's theorem (see \cite[Theorem~1.4]{MR2598208}) asserts that $\mathcal{L}$ is a completely monotone function.   While we are not able to provide a direct proof that the Boltzmann collision kernels corresponding to power-law potentials can indeed be written in the form \eqref{e:subordination}, numerical computations show that this identification holds at least up to a negligibly small error. Due to the criterion \eqref{e:kernel-comparison}, that is enough for us to obtain a useful estimate for the parameter $\Lambda_b$.

\subsection{Proofs}

The following inequality is proved in \cite{landaufisher2023} with the optimal constant $\Lambda_{local}$ in $2D$, and with $\Lambda_{local} = 4.73$ in 3D. It is the version for the projective space of an inequality that is used in the proof by Bakry and \'Emery of the log-Sobolev inequality (see \cite[Proposition~5.7.3]{bakry-2014}). The values of the parameter $\Lambda_{local}$ are improved and computed for all dimension in \cite{sehyun2024}.

\begin{thm}[The $\Gamma^2$ criterion on the projective space -- \cite{sehyun2024}] \label{t:be}
Let $f : S^{d-1} \to (0,\infty)$ be a smooth function such that $f(\sigma)=f(-\sigma)$ for every $\sigma \in S^{d-1}$. Then, the following inequality holds
\begin{equation} \label{e:be}
\Lambda_{local} \int_{S^{d-1}} \frac{|\nabla f|^2}f \dd \sigma \leq \int_{S^{d-1}} \Gamma^2_{\Delta,\Delta} (\log f, \log f) \, f(\sigma) \dd \sigma.
\end{equation}
for $\Lambda_{local} = d+3 - 1/(d-1)$. In particular $\Lambda_{local} > d$ for all $d \geq 2$.
\end{thm}

We recall the values of the entropy $\hs(f)$ and the Fisher information $\is(f)$. We also write $\js(f)$ to denote the right-hand side of the inequality in \eqref{e:be}.
\begin{align*}
\hs(f) &:= -\int_{S^{d-1}} f \log f \dd \sigma, \\
\is(f) &:= \int_{S^{d-1}} \frac{|\nabla f|^2 } f \dd \sigma, \\
\js(f) &:= \int_{S^{d-1}} \Gamma^2_{\Delta,\Delta} (\log f,\log f) \,  f(\sigma) \dd \sigma.
\end{align*}

These three quantities are related through the heat flow. Let us consider $u: [0,\infty) \times S^{d-1} \to (0,\infty)$ to be the solution to the heat equation on the sphere.
\begin{align*}
u(0,\sigma) &= f(\sigma) \\
\partial_t u &= \Delta_\sigma u.
\end{align*}
For convenience, let us use the sub-index $t$ to denote $f_t(\sigma) := u(t,\sigma)$.

\begin{lemma} \label{l:derivativesintime}
We have the following identities
\[ \partial_t \hs(f_t) = \is(f_t) \qquad \text{ and } \qquad \partial_t \is(f_t) = -2\js(f_t).\]
\end{lemma}

\begin{proof}
They are both classical. They can be verified by adapting  computations from Section \ref{s:carré-du-champ}.
\end{proof}

Our plan is to use Theorem \ref{t:be} to deduce the inequalities \eqref{e:DtoFisher} and \eqref{e:H1toNonlocal} with precise constants when $b = b_t$ is the heat kernel. Then, we integrate in time to deduce the corresponding inequalities for a generic kernel $b$ of the form \eqref{e:subordination}.

In the case that $b = h_t$, the operator $\B=\B_t$ takes the simple form
\[ \B_t f = f_t - f.\]

\begin{lemma} \label{l:using-convexity}
For any function $f:S^{d-1} \to (0,\infty)$, we have
\[ -\langle \is'(f), \B_t f \rangle \geq \is(f) - \is(f_t).\]
\end{lemma}

\begin{proof}
This is merely a consequence of the fact that the Fisher information $I$ is a convex functional.
\end{proof}

\begin{lemma} \label{l:reducingFisher}
For any function $f:S^{d-1} \to (0,\infty)$ such that $f(\sigma) = f(-\sigma)$, the following inequality holds
\[ \is(f_t) \leq \exp(-2\Lambda_{local} t) \is(f).\]
\end{lemma}

\begin{proof}
Combining Theorem \ref{t:be} with Lemma \ref{l:derivativesintime} we obtain
\[ \partial_t \is(f_t) \leq -2\Lambda_{local} \is(f_t). \]
The result follows by Gr\"onwall's inequality.
\end{proof}

In the next lemma, we justify the inequality \eqref{e:DtoFisher} with $\cOne = (1-\exp(-2\Lambda_{local} t))/2$.

\begin{lemma} \label{l:integratingBEintime}
For any function $f:S^{d-1} \to (0,\infty)$ such that $f(\sigma) = f(-\sigma)$, the following inequality holds
\[ -\langle \is'(f), \B_t f \rangle \geq (1-\exp(-2\Lambda_{local} t)) \is(f).\]
In particular, \eqref{e:DtoFisher} holds with constant $\cOne = (1-\exp(-2\Lambda_{local} t))/2$ for $\B = \B_t$.
\end{lemma}

\begin{proof}
We combine Lemmas \ref{l:using-convexity} and \ref{l:reducingFisher}. For the last observation, we are using Lemma \ref{l:fisher-along-B}.
\end{proof}

\begin{proof}[Proof of Proposition~\ref{p:cOne_for_subordinate}]
Taking into account the representation of $\Gamma^2_{\B,\Delta}$ from Lemma \ref{l:Gamma2-with-T}, we observe that for any function $g : S^{d-1} \to \R$,
\begin{align*}
\Gamma^2_{\B,\Delta} (g,g) &= \int_0^\infty \omega(t) \Gamma^2_{\B_t,\Delta} (g,g) \dd t.
\end{align*}

Applying the equality above to $g= \log f$ and recalling Lemma \ref{l:fisher-along-B}, we get
\begin{align*}
 -\langle \is'(f), \B f \rangle &= 2\int_{S^{d-1}} \Gamma^2_{\B,\Delta} (\log f,\log f) f \dd \sigma \\
 &= \int_0^\infty 2\omega(t) \int_{S^{d-1}} \Gamma^2_{\B_t,\Delta} (\log f,\log f) f \dd \sigma \dd t, \\
 \intertext{we apply Lemma \ref{l:integratingBEintime} to each $\B_t$ and get,}
 &\geq \left( \int_0^\infty \omega(t) (1-\exp(-2\Lambda_{local} t)) \dd t \right) \is(f).
 \end{align*}
 We conclude by recalling the formula in Lemma \ref{l:fisher-along-B}.
\end{proof}

The inequality \eqref{e:H1toNonlocal} is already justified for any applicable kernel in Proposition \ref{p:poincare}. In the case of kernels of the form \eqref{e:subordination}, it can be justified more easily and with an explicit formula for the constant $\cP$ in terms of $\omega$. We work it out in the next lemma.

\begin{lemma} \label{l:poincare-subordinate}
For any function $f : S^{d-1} \to (0,\infty)$ such that $f(\sigma) = f(-\sigma)$, the following inequality holds for any $t > 0$,
\begin{equation}
  \label{e:poincare-subordinate}
  \cP \int_{S^{d-1}} |\nabla f|^2 \dd \sigma \geq \iint_{S^{d-1} \times S^{d-1}} (f'-f)^2 b_t(\sigma'\cdot\sigma) \dd \sigma' \dd \sigma
\end{equation}
with \( \cP = \frac{1-e^{-2dt}}d. \)
\end{lemma}

\begin{proof}
We start by expanding the right hand side.
\begin{align*}
\iint_{S^{d-1} \times S^{d-1}} (f'-f)^2 b_t(\sigma'\cdot\sigma) \dd \sigma' \dd \sigma &= -2 \int f \B f \dd \sigma \\
&= 2 \int f (f-f_t) \dd \sigma. 
\end{align*}
We write $f$ in terms of the spectral decomposition of the Laplace-Beltrami operator $-\Delta$ on the sphere. Let $\psi_\ell$ be the normalized even eigenfunctions with eigenvalues $\lambda_1 \leq \lambda_2 \leq \lambda_3 \leq \dots$. We write
\[ f = \sum_{\ell=1}^\infty a_\ell \psi_\ell \quad \text{ and } \quad f_t = \sum_{\ell=1}^\infty a_\ell e^{-\lambda_\ell t} \psi_\ell.\]
We have
\begin{align*}
\int_{S^{d-1}} |\nabla f|^2 \dd \sigma &= \sum_{\ell=1}^\infty \lambda_\ell a_\ell^2 \\
\iint_{S^{d-1} \times S^{d-1}} (f'-f)^2 b_t(\sigma'\cdot\sigma) \dd \sigma' &= 2 \sum_{\ell=1}^\infty (1-e^{\lambda_\ell t}) a_\ell^2.
\end{align*}

The first eigenvalue of the Laplacian on the sphere $\lambda_1$ equals zero, with $\psi_1$ being a constant function. Consequently, the first term in both sums above vanishes and does not count. The second eigenvalue of the Laplacian restricted to even functions $f(\sigma) = f(-\sigma)$ is $\lambda_2 = 2d$. Then,
\[ \max_{\ell = 2,3,4,\dots} \frac{1-\exp(-\lambda_\ell t)}{\lambda_\ell} = \frac{1-\exp(-2d t)}{2d}, \]
and the result follows.
\end{proof}

Alternatively, integrating the estimate from Lemma \ref{l:poincare-subordinate} against the weight $\omega$, it is not difficult to derive Proposition~\ref{p:poincare-subordinate} from Proposition~\ref{p:poincare}. However, the proof presented above is much simpler than the proof of Proposition~\ref{p:poincare}, and thus we opted for this presentation.

\begin{proof}[Proof of Proposition~\ref{p:main-for-subordinate}]
The inequalities \eqref{e:DtoFisher} and \eqref{e:H1toNonlocal} hold due to Proposition~\ref{p:cOne_for_subordinate} and Proposition \ref{p:poincare-subordinate} with constants
\[ \cOne = \int_0^\infty \omega(t) (1-\exp(-2\Lambda_{local} t)) \dd t \qquad \text{and} \qquad \cP = \int_0^\infty \omega(t) \frac{1-e^{-2dt}}d \dd t.\]

Thus, applying Lemma \ref{l:three-inequalities}, we conclude that \eqref{e:spherical-inequality} holds with 
\[ \Lambda_b \geq 2 \cOne/\cP = d \frac{\int_0^\infty \omega(t) (1-\exp(-2\Lambda_{local} t)) \dd t}{\int_0^\infty \omega(t) (1-e^{-2dt}) \dd t}. \qedhere \]

The fact that $\Lambda_b > d$ follows from the inequality
\[ (1-\exp(-2\Lambda_{local} t)) > (1-e^{-2dt}),\]
which is a simple consequence of $\Lambda_{local} > d$.
\end{proof}

Interestingly, with the computations above we can accurately estimate the values of $\cOne$, $\cP$ and $\Lambda_b$ in the case of the fractional Laplacian.

\begin{cor}[Key constants for the fractional Laplacian] \label{c:Lambda_fractional_laplacian}
If $b_s$ is the kernel of the fractional Laplacian $(-\Delta)^s$ on the sphere, we get
\[ \cOne = \frac{(\Lambda_{local})^s}{2^{1-s}}, \qquad \cP = \frac{2^s}{d^{1-s}}, \qquad \Lambda_{b_s} \geq (\Lambda_{local})^{s} d^{1-s}.\]
\end{cor}

\begin{proof}
We apply Propositions \ref{p:cOne_for_subordinate}, \ref{p:poincare-subordinate} and \ref{p:main-for-subordinate} with the weight function $\omega(t) = c_s t^{-1-s}$, where
\[ c_s \int_0^\infty t^{-1-s} (1-e^{-t}) \dd t = 1.\]
The corollary follows after an elementay computation.
\end{proof}

\subsection{The hard-spheres case}

In the hard-spheres case $b \equiv 1$, the inequality \eqref{e:spherical-inequality} can be proved in a number of ways. Here, we obtain it by passing to the limit as $t \to \infty$ the computations for $b_t$.

\begin{prop}[Fisher is decreasing for hard spheres] \label{p:hard-spheres}
When $b \equiv 1$, the inequality \eqref{e:spherical-inequality} holds with $\Lambda_b = d$. In particular, Theorem \ref{t:main-simpler} holds in the case of 3D hard-spheres.
\end{prop}

\begin{proof}
We observe that as $t \to \infty$, the heat kernel $b_t$ converges to the constant kernel. Moreover, the parameters $\cOne$ and $\cP$ from Lemmas \ref{l:integratingBEintime} and \ref{l:poincare-subordinate} satisfy
\begin{align*}
\cOne &\to 1/2 \text{ as } t \to \infty,  \\
\cP &\to 1/d \text{ as } t \to \infty.
\end{align*}
Therefore, from Lemma \ref{l:three-inequalities}, we get $\Lambda_{b_t} \to d$ and we conclude the proof.
\end{proof}

\section{Existence of global smooth solutions}
\label{s:global_existence}

In this section, we prove Theorem \ref{t:global_existence} as a consequence of Theorem \ref{t:main}.
\bigskip

The first ingredient to establish the global existence of smooth solutions is a local-in-time existence result. Then, we combine it with a-priori estimates and a suitable conditional regularity result.

In view of the assumptions for short-time existence / blow-up (Theorem \ref{t:hst}), in order to prove that the non-cutoff Boltzmann equation with soft potentials does not blow up, we only need to establish an a-priori estimate in $\|f(t)\|_{L^{p_1}_{q_1}}$.
\bigskip

Theorem \ref{t:hst} provides a strong solution in some interval of time $[0,T]$. Provided that the initial data is in $L^\infty_q$ for sufficiently large $q$, this solution will be $C^2$ and $D^2f \in L^\infty_p$ for some arbitrary exponent $p$. Because of Lemma \ref{l:finite-fisher}, for any small time $\tau>0$, the Fisher information of $f(\tau,\cdot)$ is finite.

For suitable collision kernels $B$, Theorem \ref{t:main} tells us that the Fisher information remains bounded for all later time. In particular, since $I(f) = 4 \|\sqrt f\|_{\dot H^2}^2 \gtrsim \|f\|_{L^3}$, we know that $f \in L^\infty([\tau,T],L^3(\R^3))$ for any given $\tau>0$.

From the standard conservation of energy, we know that 
\[ \int_{\R^3} \langle v \rangle^2 f(t,v) \dd v \leq M_0 + E_0,\]
where $M_0$ and $E_0$ are the mass and energy of the initial data $f_0$.

Interpolating between these two bounds, we get that for all $t \in [\tau,T]$,
\begin{equation} \label{e:Lp-moment}
\int_{\R^3} \langle v \rangle^k f(t,v)^p \dd v \leq C_{k,p},
\end{equation}
for $p \in [1,3]$ and $k{} = 3-p$. This is enough to apply Theorem \ref{t:hst} for a range of parameters $\gamma$ and $2s$ but not for all them. Indeed, for power-law collision kernels approaching Coulomb, we would have $\gamma \to -3$ and $2s \to 2$, in which case we would need $k_1 = 5/2$ and $p_1 > 3/2$ in order to apply Theorem \ref{t:hst}, which is outside our range.
\bigskip

This restriction can easily be overcome using the propagation of higher moments. Let us recall a result from \cite{carlen2009}. In our context, it says that
\begin{thm}[{Propagation of moments -- \cite[Theorem 1(i)]{carlen2009}}] \label{t:ccl}
Any solution $f:[0,T] \times \R^d \to [0,\infty)$ of the Boltzmann equation \eqref{e:boltzmann} so that $f_0 \in L\log L \cap L^1_k$ for $k>2$, satisfies $f \in L^\infty([0,T],L^1_k)$.
\end{thm}

The precise upper bound for $\|f\|_{L^1_k}$ obtained in \cite{carlen2009} deteriorates linearly as $t \to \infty$. For the purpose of proving the global existence of smooth solutions, we only use that it does not blow up in finite time.

Let us assume that $f_0 \in L^1_{k_2}$ for some $k_2$ large enough. Applying Theorem \ref{t:ccl}, we know that the $k_2$-moment of $f$ does not blow up
\[ \sup_{t \in [0,t)} \int_{\R^3} \langle v \rangle^{k_2} f(t,v) \dd v \leq M_k.\]
Then, we can interpolate with our bound in $L^\infty(L^3)$ from the Fisher information and get \eqref{e:Lp-moment} for $k = k_1(3/2 - p/2)$. If we take $p = 3/(3+\gamma+2s)+$, we get
\[ k = k_2\left( \frac 32 - \frac 3 {2(3+\gamma+2s)} \right)-.\]
To satisfy the continuation criterion in Theorem \ref{t:hst}, we would need $k_2$ to be larger than some number depending on $\gamma$ and $2s$. For some pairs $\gamma$ and $2s$, this value of $k_2$ is smaller or equal to $2$ and it follows from the conservation of mass and energy. For some values of $\gamma$ and $2s$, $k_2$ is larger than $2$ and the result requires some further propagation of moments that is provided by Theorem \ref{t:ccl}. 


In any case, we observe that Theorem \ref{t:global_existence} follows as a consequence of Theorem \ref{t:main} together with Theorems~\ref{t:hst} and \ref{t:ccl}.

\appendix

\section{Numerical computations with collision kernels}
\label{s:numerical}

Our objective in this section is to explain our numerical tests, and report the results about collision kernels associated with inverse power-law potentials in two and three dimensions. We recall that their kernels are factorized and $\alpha (r) = r^\gamma$. In particular, Theorem~\ref{t:main} can be applied if $|\gamma|\le 2 \sqrt{\Lambda_b}$. 

We verify that these collision kernels  satisfy the inequality \eqref{e:kernel-comparison} with some reference kernel $b_0 = b_\omega$ of the form \eqref{e:subordination} for some weight function $\omega \geq 0$.

\subsection{Computation of the collision kernel for inverse power laws}

While the Boltzmann collision operator makes sense for a variety of collision kernels $B$, only some of them are derived starting from a system of interacting particles. The formulas to compute the collision kernel from a general interaction potential can be found in \cite[Chapter 1]{villani-survey} or in \cite{cercignani-1988}. They are defined after a few steps involving implicit formulas that we explain below.

In the case of a repulsive power-law potential of the form $\psi(r) = \psi_0 r^{-q+1} / (q-1)$, the collision kernel $B(r, \cos \theta)$ takes the form of the product of two functions $B = \alpha(r/2) b_{col}(\cos \theta)$, where $\alpha(r) = r^\gamma$ for $\gamma = (q-2d+1)/(q-1)$ and $b_{col}(\cos \theta)$ being a spherical kernel with a singularity at $\theta = 0$.

To explain the formula for $b_{col}(\cos \theta)$, we should first compute the deviation angle $\theta$ as a function of the impact parameter $p$ as
\[ \theta(p) := \pi - 2p \int_{r_0}^\infty \frac{\dd r}{r^2 \sqrt{1-\frac{p^2}{r^2}-\frac{4\psi_0}{r^{q-1}}}}.\]
Here, $r_0$ is defined to be the positive root of
$$ 1 - \frac{p^2}{r_0^2} - \frac{4 \psi_0}{r_0^{q-1}} = 0.$$
And then, the collision kernel $b_{col}$ is given by
$$ b_{col}(\cos \theta) = -\left( \frac{p}{\sin \theta(p)} \right)^{d-2} \frac{dp}{d\theta}. $$

The parameter $\psi_0$ can be scaled out. It becomes a multiplicative factor of the form $\psi_0^{2s}$ in front of the formula for $b_{col}(\cos \theta)$. We can assume without loss of generality that $\psi_0 = 1/4$.

The formula for $b_{col}$ described above is admittedly complicated. We cannot write an explicit formula for $b_{col}(\cos \theta)$. A few asymptotic properties can be established and are standard. We know its behavior as $\theta \to 0$. For $2s = (d-1)/(q-1)$, we have
\[ \lim_{\theta \to 0} \left\{ b_{col}(\cos \theta) \theta^{d-1+2s} \right\}= \frac 1 {q-1} \left( \sqrt{\pi} \frac{ \Gamma(q/2)}{\Gamma(q/2-1/2)} \right)^{2s}. \]

It is not too difficult to compute the values $b_{col}(\cos \theta)$ numerically, by basically following the same steps described above.

\subsection{The three dimensional case}

Let us focus on the three dimensional case: $d=3$. 

In dimension three, it can be verified that $b_{col}(\cos \theta)$ converges to a constant away from $\theta =0$ as $q \to \infty$ (and equivalently $s\to 0$). Thus, the collision kernel of inverse power-law potentials converges to the standard hard-spheres kernel as $q \to \infty$.

The fractional Laplacian $(-\Delta)^s$ on the sphere $S^2$ is defined using spectral calculus in terms of how it acts on spherical harmonics
\[ (-\Delta)^s Y_\ell^m = \lambda_\ell^s Y_\ell^m.\]

The kernel of $(-\Delta)^s$ on the sphere $S^2$, which we may write $b_{(-\Delta)^s}(\cos \theta)$, has a similar asymptotic behavior as the collision kernel $b_{col}(\cos \theta)$. There is no explicit formula for $b_{(-\Delta)^s} $either. It can be written in the form \eqref{e:subordination} with
\[ \omega_s(t) = t^{-1-s}/C_s,\]
where
\[ C_s = \int_0^\infty (1-e^{-t}) t^{-1-s} \dd t = -\Gamma(-s).\]

If we consider other weight functions $\omega(t)$ which are comparable to $\omega_s(t)$ as $t \to 0$, we obtain other kernels $b_\omega$ with the same asymptotic behavior at $\theta \to 0$ as $b_{(-\Delta)^s}$, which coincides with the asymptotic behavior of $b_{col}$. For any such kernel, we are able to precisely compute the limit
\[ \lim_{\theta \to 0} \frac{b_{col}(\cos \theta)}{b_\omega(\cos \theta)}. \]
Thus, the challenge to compute the constants $c_1$ and $C_2$ in the inequality \eqref{e:kernel-comparison} arises mostly from values of $\theta$ that are not too close to zero or $\pi$.

Computing $b_\omega$ numerically is nontrivial. All we know is that the integro-differential operator with kernel $b_\omega$ has the spherical harmonics as eigenfunctions with computable eigenvalues. For any smooth function $f$, we can compute the integro-differential operator
\[ \B_\omega f (\sigma) := \int_{S^{d-1}} (f(\sigma') - f(\sigma)) b_\omega(\sigma' \cdot \sigma) \dd \sigma',\]
by decomposing the function $f$ as a sum of spherical harmonics and multiplying each term by the corresponding eigenvalue. We must reverse engineer the kernel $b_\omega$ from the values of this operator. A simple way to achieve it is by considering a function $f$ which is smooth but it is very concentrated around $\sigma' \cdot e_1 = x_0$. From the value of $\B_\omega f(e_1)$, we estimate the value of $b_\omega(x_0)$. To implement this computation with reasonable accuracy, we need to use a large resolution and many modes. This method works farily well, especially if $\theta$ is not too close to zero. As we discussed before, we already know the approximate value of the ratio $b_{col}/b_\omega$ when $\theta$ is almost zero.

Using the ideas above, we ran numerical simulations to estimate the ratio of $c_1/C_2$ for the inequality \eqref{e:kernel-comparison}. When we compared the collision kernel $b_{col}$ with the kernel of the spherical fractional Laplacian, we observed that $c_1/C_2$ is approximately $0.71$ for $s = 3/4$. The ratio improves as $s \to 1$. We experimented with reference kernels of the form \eqref{e:subordination} with different weight functions $\omega(t)$ of the form $\omega(t) = t^{-1-s} \omega_1(t)$, for some nonnegative function $\omega_1$ so that $\omega_1(0)=1$.

The kernel $b_\omega$ matches the collision kernel of power-law potentials $b_{col}$ very closely for $s \in [3/4,1)$ with the following choice of weight function
\begin{equation}  \label{e:omega_guess}
\omega(t) := \left\{ 1 - \min\left( \frac{13}8 - \frac 3 2 s, \frac 4 {10} \right) (1-\exp(-2t)) \right\} t^{-1-s}.
\end{equation}
In this case, we obtained that \eqref{e:spherical-inequality} holds for $2s \in [1.5, 2)$ with $c_1/C_2 > 0.95$. The following table summarizes some of the results we got from the numerical computations for various values of the exponent $q \in (2,7/3]$. The computed lower bound for $2\sqrt{\Lambda_b}$ is confortably above the value of $|\gamma|$ in every case.
\begin{center}
\begin{tabular}{|r|r|r|r|r|}
\hline
$q$ & $2s$ & $\gamma$ & $c_1/C_2$ & lower bound for $2\sqrt{\Lambda_b}$ \\
\hline
2.01 & 1.98 & -2.96 & 0.98 & 4.64 \\
2.05 & 1.90 & -2.81 & 0.99 & 4.59 \\
2.10 & 1.82 & -2.64 & 0.99 & 4.54 \\
2.25 & 1.60 & -2.20 & 0.97 & 4.38 \\
2.33 & 1.50 & -2.00 & 0.98 & 4.36 \\
\hline 
\end{tabular}
\end{center}

The weight function \eqref{e:omega_guess} was obtained by guessing, trial and error. We do not claim that $b_{col}$ is an exact scalar multiple of $b_\omega$ for the weight $\omega$ given by \eqref{e:omega_guess}. All that numerical simulations show is that these two kernels are very close to each other, more than enough of what we need to apply the inequality \eqref{e:spherical-inequality}. It is possible that a further adjustment of the weight $\omega$ produces an even better approximation of $b_{col}$. It is tempting to believe that there exists some weight $\omega$ for which $b_\omega = b_{col}$. For now, that weight seems very difficult to find.

The lower bounds for $2\sqrt{\Lambda_{b_\omega}}$ in the table above are computed using Proposition \ref{p:main-for-subordinate}, for the kernel $\omega$, and Proposition \ref{p:perturbation} with the value of $c_1/C_2$ obtained numerically.

The code for our numerical experiments in 3D, with appropriate documentation, can be found in \cite{numerics}.

\subsection{The two dimensional case}

In 2D, we started with basic numerical computations comparing the kernel $b_{col}$ for inverse power-law collision operators with the kernel $b_s$ of the fractional Laplacian. In this case, when $q > 2$, we have $s \in (0,1/2]$ and $\gamma \in (-1,1)$. For $q \in (3/2,2)$, we have $2s \in (1/2,1)$ and $\gamma \in (-3,-1)$. The optimal ratio $c_1/C_2$ for which \eqref{e:kernel-comparison} holds seems to equal one for $q = 2$ ($s \to 1/2$) and changes monotonically with respect to $q$ in both directions. For $s=0.1$ (which corresponds to $q=11$), we get $c_1/C_2 \approx 2/3$. We use Corollary \ref{c:Lambda_fractional_laplacian} to estimate the parameter $\Lambda_{b_s}$ and \eqref{e:kernel-comparison} to get a lower bound for the kernel $b$ in each case. The lower bound we obtain for $\Lambda_b$ with this basic computation is good enough to apply our Theorem \ref{t:main} except when $q$ is very close to $3/2$. The assumption of Theorem \ref{t:main} is satisfied by a confortable margin certainly for $q>2$. The following table summarizes the results of the computation.
\begin{table}[ht!] 
  \centering
\begin{tabular}{|r|r|r|r|r|r|}
\hline
$q$ & $2s$ & $\gamma$ & $c_1/C_2$ & lower bound for $2\sqrt{\Lambda_b}$ \\
\hline
1.51 & 1.96 & {\bf -2.92} & 0.44 & {\bf 2.65} \\
1.55 & 1.82 & -2.64 & 0.52 & 2.80 \\
1.67 & 1.50 & -2.00 & 0.72 & 3.10 \\
1.75 & 1.33 & -1.67 & 0.82 & 3.23 \\
1.90 & 1.11 & -1.22 & 0.95 & 3.34 \\
2.00 & 1.00 & -1.00 & 1.00 & 3.36 \\
2.10 & 0.91 & -0.82 & 0.96 & 3.25 \\
4.00 & 0.33 & 0.33 & 0.84 & 2.75 \\
10.00 & 0.11 & 0.78 & 0.68 & 2.39 \\
20.00 & 0.05 & 0.89 & 0.59 & 2.21 \\
40.00 & 0.03 & 0.95 & 0.53 & 2.09 \\
\hline
\end{tabular}
\label{t:2D}
\end{table}

It is interesting to remark that as $q \to \infty$ and $s \to 0$, the operator $\B$ corresponding to $b_{col}$ converges to the 2D hard-spheres model. The value of $\cP$ in Proposition \ref{p:poincare} converges to the corresponding one for hard spheres and also does the lower bound for $b_{col}$. Thus, the limit of $\Lambda_b$ as $s \to 0$ cannot be smaller than the one for hard spheres, which is $\sqrt{2}$. This will not necessarily be reflected by the numerical computations, because they do not rely on the formula of Proposition \ref{p:poincare} to estimate $\cP$. They are based purely on pointwise upper and lower bounds between the kernel $b_{col}$ and the reference kernel $b_0$ of the fractional Laplacian. This analysis shows, however, that the assumption of Theorem \ref{t:main} will surely hold for power-law potentials in 2D when $q$ is large enough.

Note also that the hard spheres kernel in 2D (which coincides with the limit of the power-law kernels as $q \to \infty$) has the explicit formula $b(\cos \theta) + b(\cos (\pi-\theta)) = \sin(\theta/2) + \cos(\theta/2)$. For $s$ small (and $q$ large), the power-law collision kernels in 2D will not be monotone decreasing for $\theta \in [0,\pi/2]$. Thus, there is no hope to write them exactly in the form \eqref{e:subordination}.

For $q$ close to $3/2$, the computation comparing the collision kernel with the fractional Laplacian are not good enough to ensure the assumption of Theorem \ref{t:main}. Thus, we will use an ad-hoc weight function $\omega$ as we did for the 3D case. We compared numerically the collision kernel for power-law potentials in 2D with $q \in [3/2,2]$ with the collision kernel of the form \eqref{e:subordination} with the weight 
\[ \omega(t) = t^{-1-s} \left\{ 1 + 2 (2s-1)^2 (1-\exp(-2t)) \right\}. \]
We got the results summarized in the following table. \medskip

\begin{center}
\begin{tabular}{|r|r|r|r|r|r|}
\hline
$q$ & $2s$ & $\gamma$ & $c_2/C_1$ & lower bound for $2\sqrt{\Lambda_b}$ \\
\hline
1.99 & 1.01 & -1.02 & 1.00 & 3.36 \\
1.90 & 1.11 & -1.22 & 0.96 & 3.36 \\
1.80 & 1.25 & -1.50 & 0.95 & 3.41 \\
1.70 & 1.43 & -1.86 & 0.97 & 3.52 \\
1.60 & 1.67 & -2.33 & 0.96 & 3.63 \\
1.55 & 1.82 & -2.64 & 0.94 & 3.69 \\
1.51 & 1.96 & -2.92 & 0.93 & 3.79 \\
1.501 & 1.996 & -2.99 & 0.93 & 3.82 \\
\hline
\end{tabular}
\end{center}
\medskip

In all cases, we get lower bounds for $\Lambda_b$ for which our assumption for Theorem \ref{t:main} is satisfied by a confortable margin.

\section{Vector fields generating rotations}
\label{s:vector-fields}

Following \cite{landaufisher2023}, we construct a family of vector fields $\bb_1, \dots, \bb_N$, for $N=d(d-1)/2$ that are tangential to the sphere $S^{d-1}$ and generate the Lie algebra $so(d)$. They can be used as an alternative to the transformations $P_{\sigma,\sigma'}$ described in Section \ref{s:maps-between-tangent-spaces} to derive formulas that involve $\Gamma^2_{\Delta,\B}$ and $\nabla \B f$.

For $d=2$, we would consider just the single vector field $\bb_1(\sigma) = \sigma^\perp$. 

In three dimensions $d=3$ it is the set of three vector fields used in \cite{landaufisher2023}.
\begin{equation} \label{e:vectorsb}
\bb_1(\sigma) = \begin{pmatrix} 0 \\ -\sigma_3 \\ \sigma_2 \end{pmatrix}, \qquad \bb_2(\sigma) = \begin{pmatrix} \sigma_3 \\ 0 \\ -\sigma_1 \end{pmatrix}, \qquad \bb_3(\sigma) = \begin{pmatrix} -\sigma_2 \\ \sigma_1 \\ 0 \end{pmatrix}.
\end{equation}

In higher dimension $d \geq 4$, there are $d(d-1)/2$ vector fields that general the Lie group $so(d)$. They can be written in terms of pairs of indices $1 \leq i < j \leq d$ as 
\[ \bb_{ij}(\sigma) = \sigma_i e_j - \sigma_j e_i.\]
The flow of the vector field $\bb_{ij}$ consists in a rotation of the sphere on the $ij$-plane. In particular, the flow of any of these vector fields is an isometry, and the vector fields are divergence free.

By reindexing if necessary, we write these vector fields as $\bb_1,\dots,\bb_N$.

The following simple lemma is justified in \cite{landaufisher2023}.

\begin{lemma}[Gradient and Laplacian through the $\bb_i$'s] \label{l:vector-fields}
For any smooth functions $f,g:S^{d-1} \to \R$, we have
\begin{align*}
\Delta f &= \sum_{i=1}^N \bb_i \cdot \nabla (\bb_i \cdot \nabla f), \\
|\nabla f|^2 &= \sum_{i=1}^N (\bb_i \cdot \nabla f)^2, \\
\nabla f \cdot \nabla g &= \sum_{i=1}^N (\bb_i \cdot \nabla f)(\bb_i \cdot \nabla g).
\end{align*}
\end{lemma}

With the help of these vector fields, we can provide an explicit expression for $\Gamma^2_{\B,\Delta}(f,f)$.

\begin{lemma}[Iterated carré du champ through the $\bb_i$'s] \label{l:Gamma2-with-vector-fields}
For any smooth function $f : S^{d-1} \to \R$, we have,
\[ \Gamma^2_{\B,\Delta}(f,f) = \frac 12 \sum_{i=1}^N \int_{S^{d-1}} \bigg( b_i(\sigma') \cdot \nabla f(\sigma') - b_i(\sigma) \cdot \nabla f(\sigma)\bigg)^2 b (\sigma' \cdot \sigma) \dd \sigma'. \]
In particular, $\Gamma^2_{\B,\Delta}(f,f) \geq 0$.
\end{lemma}

We first prove the following technical lemmas.

\begin{lemma} \label{l:spherical-shifts}
Given $b : \R \to \R$, consider the function from $S^{d-1} \times S^{d-1}$ given by $(\sigma,\sigma') \mapsto b(\sigma \cdot \sigma')$. Then, for any of the vector fields $\bb_k$ on $TS^{d-1}$, we have
\[ \bb_k(\sigma) \cdot \nabla_\sigma [b(\sigma \cdot \sigma')] = - \bb_k(\sigma') \cdot \nabla_{\sigma'} [b(\sigma \cdot \sigma')].\]
\end{lemma}

\begin{proof}
We start with the explicit computations
\begin{align*}
\nabla_\sigma [b(\sigma \cdot \sigma')] &= \db(\sigma'\cdot \sigma) \ \sigma',\\
\nabla_{\sigma'} [b(\sigma \cdot \sigma')] &= \db(\sigma'\cdot \sigma) \ \sigma.
\end{align*}

Since $\bb_k(\sigma) \in T_{S^{d-1}}(\sigma)$, then $\bb_k(\sigma) \perp \sigma$. Thus, 
\begin{align*}
\bb_k(\sigma) \cdot \nabla_\sigma [b(\sigma \cdot \sigma')] &= \db(\sigma'\cdot \sigma) \ \bb_k(\sigma) \cdot \sigma' \\
&= \db(\sigma'\cdot \sigma) \ \bb_k(\sigma) \cdot (\sigma'-\sigma).
\end{align*}

The vector field $\bb_k(\sigma)$ is given by $\bb_k(\sigma) = A_k\sigma$ for some antisymmetric matrix $A_k$. In particular, it extends as a linear operator in $\R^d$ and
$\bb_k(\sigma) - \bb_k(\sigma') \perp \sigma - \sigma'$. Therefore $\bb_k(\sigma) \cdot (\sigma'-\sigma) = \bb_k(\sigma') \cdot (\sigma'-\sigma)$, and we get
\begin{align*}
\bb_k(\sigma) \cdot \nabla_\sigma [b(\sigma \cdot \sigma')] &= \db(\sigma'\cdot \sigma) \ \bb_k(\sigma') \cdot (\sigma'-\sigma) \\
&= -\db(\sigma'\cdot \sigma) \ \bb_k(\sigma') \cdot \sigma \\
&= -\bb_k(\sigma') \cdot \nabla_{\sigma'} [b(\sigma \cdot \sigma')]. \qedhere
\end{align*}
\end{proof}

\begin{lemma} \label{l:flowing-B-with-bi}
For any of the vector fields $\bb_i$ on $TS^{d-1}$, the following expression holds
\[ \bb_i(\sigma) \cdot \nabla_\sigma \B f = \int_{\sigma' \in S^{d-1}} (\bb_i(\sigma') \cdot \nabla f' - \bb_i(\sigma) \cdot \nabla f) b(\sigma'\cdot \sigma) \dd \sigma'.\]
\end{lemma}

\begin{proof}
Without loss of generality, we assume that the kernel $b$ is smooth. The general case follows by a standard approximation procedure.

We differentiate within the integral expression for $\B$
\begin{align*}
\bb_i(\sigma) \cdot \nabla_\sigma \B f &= \bb_i(\sigma) \cdot \left( \int (f'-f) b(\sigma' \cdot \sigma) \dd \sigma' \right) \\
&= \int -\bb_i(\sigma) \cdot \nabla f \ b(\sigma' \cdot \sigma) + (f'-f) \bb_i(\sigma) \cdot \nabla_\sigma [b(\sigma' \cdot \sigma)] \dd \sigma' \\
\intertext{We apply Lemma \ref{l:spherical-shifts} for the second term.}
&= \int -\bb_i(\sigma) \cdot \nabla f \ b(\sigma' \cdot \sigma) - (f'-f) \bb_i(\sigma') \cdot \nabla_{\sigma'} [b(\sigma' \cdot \sigma)] \dd \sigma' \\
\intertext{We integrate by parts the second term recalling that $\bb_i(\sigma')$ is divergence free.}
&= \int (\bb_i(\sigma') \cdot \nabla f' -\bb_i(\sigma) \cdot \nabla f) \ b(\sigma' \cdot \sigma) \dd \sigma'. \qedhere
\end{align*}
\end{proof}

We can now prove Lemma~\ref{l:Gamma2-with-vector-fields}.

\begin{proof}[Proof of Lemma~\ref{l:Gamma2-with-vector-fields}]
By definition
\begin{align*}
\Gamma_{\B,\Delta}^2(f,f) &= \frac 12 \B \Gamma_\Delta(f,f) - \Gamma_\Delta(f,\B f) \\
&= \frac 12 \int \left( |\nabla f(\sigma')|^2 - |\nabla f(\sigma)|^2 \right) b(\sigma'\cdot\sigma) \dd \sigma' - \nabla f(\sigma) \cdot \nabla\left( \int (f'-f) b(\sigma'\cdot \sigma) \dd \sigma' \right), \\
\intertext{using Lemma \ref{l:vector-fields},}
&= \sum_i \frac 12 \int \left( (b_i(\sigma') \cdot \nabla f')^2 - (b_i(\sigma) \cdot \nabla f)^2 \right) b(\sigma'\cdot\sigma) \dd \sigma' \\ &\qquad - \bigg( b_i(\sigma) \cdot \nabla f(\sigma) \bigg)\left( b_i(\sigma) \cdot \nabla \int (f'-f) b(\sigma'\cdot \sigma) \dd \sigma' \right), \\
\intertext{we apply Lemma \ref{l:flowing-B-with-bi} for the second term,}
&= \sum_i \frac 12 \int \left( (b_i(\sigma') \cdot \nabla f')^2 - (b_i(\sigma) \cdot \nabla f)^2 \right) b(\sigma'\cdot\sigma) \dd \sigma' \\ &\qquad - \bigg( b_i(\sigma) \cdot \nabla f(\sigma) \bigg) \left( \int \big(b_i(\sigma') \cdot \nabla f'- b_i(\sigma) \cdot \nabla f \big) b(\sigma'\cdot \sigma) \dd \sigma' \right) \\
&= \sum_i \frac 12 \int \left( b_i(\sigma') \cdot \nabla f' - b_i(\sigma) \cdot \nabla f \right)^2 b(\sigma'\cdot\sigma) \dd \sigma'. \qedhere
\end{align*}
\end{proof}

\bibliographystyle{plain}
\bibliography{boltzmannfisher}

\end{document}